\documentclass[11pt]{amsart}

\usepackage{epsf,amssymb,amsmath,overpic,amscd,psfrag,stmaryrd,times,color,mathrsfs,euscript,enumitem}
\usepackage[all]{xy}
\usepackage{hyperref}
\usepackage{marginnote}
\usepackage{graphicx}

\usepackage{amsfonts}
\usepackage{amsmath}
\usepackage{soul}
\usepackage{graphicx}

\usepackage{tikz,tikz-cd}

\newtheorem{thm}{Theorem}[section]
\newtheorem{prop}[thm]{Proposition}
\newtheorem{lemma}[thm]{Lemma}
\newtheorem{cor}[thm]{Corollary}

\newtheorem{claim}[thm]{Claim}
\newtheorem{fact}[thm]{Fact}

\numberwithin{equation}{subsection}
\numberwithin{thm}{subsection}

\theoremstyle{definition}
\newtheorem{defn}[thm]{Definition}

\newtheorem{notation}[thm]{Notation}

\theoremstyle{remark}

\newtheorem{rmk}[thm]{Remark}
\newtheorem{example}[thm]{Example}

\DeclareMathAlphabet{\mathpzc}{OT1}{pzc}{m}{it}

\renewcommand{\H}{\mathbb{H}}
\newcommand{\C}{\mathbb{C}}

\newcommand{\R}{\mathbb{R}}
\newcommand{\Z}{\mathbb{Z}}
\newcommand{\Q}{\mathbb{Q}}
\newcommand{\N}{\mathbb{N}}
\newcommand{\V}{\mathbb{V}}
\newcommand{\nV}{V}
\newcommand{\K}{\mathbb{K}}

\newcommand{\bdry}{\partial}
\newcommand{\s}{\vskip.1in}
\newcommand{\n}{\noindent}

\newcommand{\E}{\mathbb{E}}

\newcommand{\be}{\begin{enumerate}}
\newcommand{\ee}{\end{enumerate}}
\newcommand{\op}{\operatorname}

\newcommand{\Kur}{\mathscr{K}}

\newcommand{\cb}{\color{black}}
\newcommand{\cbu}{\color{black}}

\begin{document}

\title[Equivariant Lagrangian Floer cohomology]
{Equivariant Lagrangian Floer cohomology via semi-global Kuranishi structures}

\author{Erkao Bao}
\address{Simons Center for Geometry and Physics, State University of New York, Stony Brook, NY 11790}
\email{baoerkao@gmail.com} 

\author{Ko Honda}
\address{University of California, Los Angeles, Los Angeles, CA 90095}
\email{honda@math.ucla.edu} \urladdr{http://www.math.ucla.edu/\char126 honda}

\date{This version: April 22, 2020.}



\begin{abstract}
Using a simplified version of Kuranishi perturbation theory that we call {\em semi-global Kuranishi structures,} we give a definition of the equivariant Lagrangian Floer cohomology of a pair of Lagrangian submanifolds that are fixed under a finite symplectic group action and satisfy certain simplifying assumptions.
\end{abstract}

\maketitle

\setcounter{tocdepth}{1}
\tableofcontents

\section{Introduction} \label{section: introduction}

Let $G$ be a finite group. The equivariant Lagrangian Floer cohomology for a pair of Lagrangians fixed under a symplectic $G$-action was first defined and studied in \cite{KS} and later in \cite{SS, He1, He2, He3, HLS}. One of the main difficulties in defining such a theory is achieving transversality of the moduli spaces of $J$-holomorphic curves using an equivariant almost complex structure $J$. Indeed, there are obstructions to the existence of equivariant regular almost complex structures; see \cite{KS, SS}. The paper \cite{HLS} uses an infinite family of non-equivariant regular almost complex structures and an algebraic approach to define equivariant cohomology.

The goal of this paper is to give an alternate definition of equivariant Lagrangian Floer cohomology using an equivariant almost complex structure $J$ that is not necessarily regular.   This involves constructing an equivariant version of a semi-global Kuranishi structure, which is a simplified version of the Kuranishi structures of \cite{FOn, FO3} used in \cite{BH2}; compare to \cite{MW} for the Kuranishi atlas formulation. It is worth mentioning that there is a construction of equivariant Kuranishi charts in \cite{Fu} in a more general situation via a quite different approach. 

\s
Let $(M,\omega)$ be a compact symplectic manifold of dimension $2n$, and let $L_0$ and $L_1$ be oriented Lagrangian submanifolds of $M$ that intersect transversely. Suppose $G$ acts on $(M,\omega)$ symplectically and satisfies $g(L_i)=L_i$ for all $g\in G$ and $i=0,1$; and that $G$ fixes the orientations of $L_i$.

\cbu
We make the following simplifying assumption:
\begin{enumerate}[label = \textbf{(S)}]
\item \label{condition: S}the maps $\pi_2(M)\stackrel{\int \omega}\to \R$ and $\pi_2(M,L_i)\stackrel{\int\omega}\to \R$ for $i=0,1$ have image $0$.
\end{enumerate}
More informally, \ref{condition: S} says that for all almost complex structures we consider we want to avoid disk and sphere bubbles.  

We also assume that either $M$ is closed or $M$ has {\em contact type boundary}, i.e., on a neighborhood of $\bdry M$ there exists a $1$-form $\sigma$ such that $\omega=d\sigma$ and the vector field $X$ defined by $\iota_X\omega=\sigma$ is positively transverse to $\bdry M$. Note that because $\omega$ is $G$-invariant, 
by averaging $\sigma$ over $G$, we can take $\sigma$ to be $G$-invariant.

The time-$r$ flow $\phi^r$ of $X$ gives a diffeomorphism $\Phi$ from $(-\epsilon, 0]\times\partial M$ to a neighborhood of $\partial M$ defined by $(r,m) \mapsto \phi^r(m)$. 
Since $$\mathcal L_X \sigma = d \iota_X \sigma + \iota_X d \sigma = d \iota_X \iota_X \omega + \iota_X \omega = \sigma,$$
we have $(\phi^r)^* \sigma = e^r \sigma$. Setting $\alpha = \sigma|_{\partial M}$, we obtain $\Phi^* \sigma = e^r \alpha$. Since $\alpha \wedge (d\alpha)^{n-1} = (\iota_X \omega^{n})|_{\partial M}$ and $X$ is transverse to $\partial M$, 
$\alpha \wedge (d\alpha)^{n-1}$ is a volume form on $\partial M$, and hence $\alpha$ is a contact form. 
We denote by $\xi=\ker \alpha$ the contact structure and $R_\alpha$ the Reeb vector field of $\alpha$.

Let $J$ be a $G$-invariant, $\omega$-compatible almost complex structure on $M$, i.e., $\omega(\cdot, J\cdot)$ is a $G$-invariant Riemannian metric. Near $\partial M$ we assume that $J$ is convex at the boundary.  More specifically:
\begin{enumerate} [label = \textbf{(J)}]
\item \label{condition: J} on the collar neighborhood $( (-\epsilon,0]\times \bdry M,e^r\alpha)$, $J$ is compatible with $d \alpha$ and maps $\xi$ to $\xi$ and $\partial_r$ to the Reeb vector field $R_\alpha$ of $\alpha$.
\end{enumerate}

\s\n
{\em Exact case.} One special case for which \ref{condition: S} holds is:
\begin{itemize}
\item $(M,\omega=d\sigma)$ is a {\em Liouville domain}, i.e., $M$ is compact and the Liouville vector field $X$ defined by $\iota_X\omega=\sigma$ points out of $\partial M$;
\item $L_0$ and $L_1$ are compact {\em exact} Lagrangians in $M$ with Legendrian boundary, where exactness means that $\sigma|_{L_i}$ is an exact $1$-form on $L_i$ for $i=0,1$.
\end{itemize}
\cb

To orient the relevant moduli spaces of $J$-holomorphic strips, following \cite[Section 8.1]{FO3} we assume that:
\begin{enumerate}[label = \textbf{(O)}]
\item \label{condition: O} the pair $(L_0,L_1)$ is equipped with a relative spin structure which is preserved by $G$.
\end{enumerate}
\cbu
See Section~\ref{section: orientations} for more details on the auxiliary orientation data including relative spin structures,
and Section~\ref{relative spin under G} for the notion of a $G$-invariant relative spin structure. \cb
 In particular we assume that $L_0$ and $L_1$ are oriented.  (See Seidel~\cite{Se} for orientations using Pin structures and Solomon~\cite{So} for orientations using relative Pin structures.)

\cbu
We denote by $R$ the Novikov ring 
$$\Big\{\sum_{i=0}^\infty a_i T^{\lambda_i} ~|~  a_i \in \Z, \lambda_i \in \R^{\geq 0}, \lambda_0 = 0 ~~~\text{ and }  \lim_{i\to \infty} \lambda_i = \infty \Big\},$$
where $T$ is a formal parameter. \cb

The Lagrangian Floer cochain complex $CF^{\bullet}(L_0,L_1)$ of the pair $(L_0,L_1)$ is the free module over the coefficient ring $R$ generated by $L_0\cap L_1$ with differential $d$ whose definition we give below.

\cbu For $p,q\in L_0\cap L_1$, let $\pi_2(p,q)$ be the set of homotopy classes of continuous maps $u:[0,1]\times[0,1]\to M$ with boundary conditions $u(0,t)=q$, $u(1,t)=p$, $u(s,0)\in L_0$, $u(s,1)\in L_1$. \cb
Let $\widetilde{\mathcal M}_J(p,q;A)$, where $A\in \pi_2(p,q)$, be the space of smooth maps $u:\mathbb R \times [0,1] \to M$ that satisfy:
\begin{enumerate}[label = (A\arabic*)]
\item \label{A1} $\overline{\bdry}_J u:=u_s+J(u)u_t=0$;
\item \label{A2} $u|_{\R \times \{i\}} \subseteq L_i$ for $i\in \{0,1\}$;
\item \label{A3} $\displaystyle \lim_{s\to -\infty}u(s,t)=q$ and $\displaystyle\lim_{s\to +\infty} u(s,t)=p$; and
\item \label{A4} $[u]=A$.
\end{enumerate}
Note that $\R$ acts on $\widetilde{\mathcal M}_J(p,q;A)$ by translation in the domain and we denote
$$\mathcal M_J(p,q;A):=\widetilde{\mathcal M}_J(p,q;A)/\R.$$
We also denote the virtual (= expected) dimension of $\mathcal{M}_J(p,q;A)$ by $\op{vdim}\mathcal{M}_J(p,q;A)$.

\begin{notation}
We use the notation $\overline{\mathcal{M}}_J(p,q;A)$ to mean the space of possibly broken strips from $p$ to $q$ in the class $A$ and
$$\bdry \mathcal{M}_J(p,q;A)=\overline{\mathcal{M}}_J(p,q;A)-\mathcal{M}_J(p,q;A).$$
\end{notation}

Suppose for the moment that $J$ is regular. Then we have the differential
\begin{gather*}
d: CF^{\bullet}(L_0,L_1)\to CF^{\bullet}(L_0,L_1)\\
d[p]=\sum_{q\in L_0\cap L_1, A\in \pi_2(p,q)}\# \mathcal M_J(p,q;A) \cdot T^{\int_A\omega} [q],
\end{gather*}
where $\# \mathcal M_J(p,q;A)=0$ if $\op{vdim} \mathcal M_J(p,q;A) \neq 0$. \cbu
Note that for each $\lambda \geq 0$, the number of $(q, A)$ such that $\omega(A) \leq \lambda$ and  $\mathcal M_J(p,q;A) \neq \emptyset$ is finite. \cb
Then as usual one shows that $d^2 = 0$ and defines the usual Lagrangian Floer homology by $HF^{\bullet}(L_0,L_1) := \ker d / \op{Im} d$.

We recall the definition of equivariant cohomology of a space $Y$ with a $G$-action. Let $BG$ be the classifying space of $G$ and let $EG$ be the universal bundle over $BG$. The diagonal action of $G$ on $EG \times Y$ is free and the quotient is denoted by $EG \times_G Y$. The $G$-equivariant cohomology of $Y$ with coefficient ring $R$ is defined to be $H^{\bullet}(EG\times _G Y; R)$. Let $C_{\bullet}(A)$ be the singular chain complex of the space $A$ over $R$ and $C^{\bullet}(A) = \op{Hom}_{R}(C_{\bullet}(A), R)$ be the singular cochain complex of $A$. Since the singular chain complexes and cochain complexes of $EG$ and $Y$ are invariant under the $G$-action and their boundary maps are $G$-equivariant, they can be viewed as complexes over the group ring $R[G]$. Then we have
\begin{align*}
H^{k}(EG\times _G Y; R) & \cong H^k(\op{Hom}_{R}(C_\bullet(EG\times _G Y), R)) \\
 & \cong H^k(\op{Hom}_{R}(C_\bullet(EG)\otimes _{R[G]} C_\bullet(Y), R)) \\
 & \cong H^k(\op{Hom}_{R[G]}(C_\bullet(EG), C^\bullet(Y))).
\end{align*}
In the second and third lines we are taking the $k$-th cohomology of the {\em total complex} of a double complex.  We are also viewing $C_\bullet (EG)$ as a complex of right $R[G]$-modules and $C_\bullet (Y)$ as a complex of left $R[G]$-modules. \cbu
We can also take a smaller model for $C_\bullet (EG)$: The projective resolution $P_\bullet$ of $R$ over $R[G]$ is chain homotopic to $C_\bullet (EG)$ and 
\begin{equation} \label{eqn: def of equiv}
H^{k}(EG\times _G Y; R) \cong H^k(\op{Hom}_{R[G]}(P_\bullet, C^\bullet(Y))).
\end{equation}\cb

Returning to the \cbu ``usual" definition of equivariant Lagrangian Floer cohomology assuming $J$ is regular, \cb we  replace $C^{\bullet}(Y)$ by $CF^{\bullet}(L_0, L_1)$ in Equation~\eqref{eqn: def of equiv}.
More precisely, since $J$ is invariant under the $G$-action, we have
$$\# \mathcal M_J(p,q;A)=\# \mathcal M_J(g(p),g(q); g(A))$$
for all $g\in G$. Hence  $d$ is a $R[G]$-linear map on $CF^{\bullet}(L_0, L_1)$.
We can then define the $G$-equivariant Lagrangian Floer cohomology group $HF^{\bullet}_G(L_0,L_1)$ as the cohomology of the total complex of \cbu $\op{Hom}_{R[G]}(P_{\bullet}, CF^{\bullet}(L_0, L_1))$. \cb

\begin{example}\cite{FO}
Let $f: L_0 \to \R$ be a $G$-equivariant Morse function and let $L_1$ be $\op{graph} (\epsilon \cdot df) \subset T^*L_0$ for some small $\epsilon>0$. Then $HF^{\bullet}_G(L_0, L_1) \cong H^{\bullet}_G(L_0; R)$.
\end{example}

In general, a \cbu $G$-invariant $J$ \cb is not regular and the moduli space $\mathcal M_J(p,q;A)$ is not transversely cut out.  The main \cbu contribution \cb of this paper is to obtain a $G$-equivariant cochain complex $CF^{\bullet}(L_0, L_1)$ when $J$ is not regular by constructing an equivariant version of a semi-global Kuranishi structure, initially developed in \cite{BH2} for contact homology.  The equivariant semi-global Kuranishi structure comes with a section ${\mathfrak s}$, and while the Kuranishi structure itself is $G$-equivariant, the section is not. This creates some difficulties, but interestingly enough there is a perturbed count of $\#\mathcal M_J(p,q;A)$ that still remains $G$-invariant (cf.\ Theorem~\ref{thm: equivariant count}).

Our main theorem is therefore the following:

\begin{thm}[Equivariant Lagrangian Floer cohomology] \label{thm: main}
Suppose $G$ acts on $(M,\omega)$ symplectically and for each $i=0,1$, $L_i$ is oriented, $g(L_i)=L_i$, for each $g\in G$, and $G$ fixes the orientation of $L_i$. \cbu
If \ref{condition: S} and \ref{condition: O} hold, then there exists an $R$-module $HF_{G}^{\bullet}(L_0,L_1)$ which is an invariant of $(L_0,L_1)$ under $G$-equivariant Hamiltonian isotopy. 
Moreover, when there exists a regular $G$-invariant $\omega$-compatible almost complex structure on $M$ satisfying \ref{condition: J}, the usual definition of equivariant Lagrangian Floer cohomology can be made and agrees with $HF_{G}^{\bullet}(L_0,L_1)$.
\cb
\end{thm}

If we want to equip the Lagrangian Floer homology groups with a $\Z$-grading, we assume that $(L_0, L_1)$ is a $G$-equivariant graded Lagrangian pair; see Section ~\ref{subsection: grading} for details.

The definition of $HF_G^{\bullet}(L_0,L_1)$ is given in Section~\ref{subsection: chain complex} and its invariance under $G$-equivariant Hamiltonian isotopy is given in Section~\ref{subsection: chain homotopy}. Most of the work is devoted to the construction of the semi-global Kuranishi structure in Section~\ref{section: equivariant semi-global Kuranishi structure} and the equivariance of the curve count in Section~\ref{section: equivariance}. The agreement with the usual definition for regular $J$ is automatic.

\s\n
{\em Acknowledgements.} We thank Kristen Hendricks, Robert Lipshitz, and Sucharit Sarkar for explaining to us their approach to equivariant Lagrangian Floer cohomology in \cite{HLS}. The first author thanks Garrett Alston and Cecilia Karlsson for discussions on orientations and Vincent Colin for providing him a great visiting opportunity at the Lebesgue Center of Mathematics and the Universit\'e de Nantes, where part of this work was carried out. The first author also thanks the Simons Center for Geometry and Physics, where he worked on this paper.

\section{Equivariant semi-global Kuranishi structure} \label{section: equivariant semi-global Kuranishi structure}

The construction of the equivariant semi-global Kuranishi structure follows the same steps as that of \cite{BH2}. The only differences are that (i) we consider sections, not multisections, and (ii) we pay attention to $G$-equivariance.

\subsection{$G$-invariant almost complex structure}
The following lemma is well-known.
\begin{lemma}\label{lemma: J}
\cbu There exists an almost complex structure $J$ which is $\omega$-compatible, $G$-invariant, and satisfies \ref{condition: J} if $\bdry M\not=\varnothing$.\cb
\end{lemma}

\begin{proof}
\cbu If $\bdry M\not=\varnothing$, then on the collar neighborhood $U=(-\epsilon,0]\times \bdry M$, $\partial_r$, $R_\alpha$, and $\xi$ are preserved by $G$.  Choose a Riemannian metric $\hat g$ on $M$ such that 
\begin{enumerate}[label=\textbf{($\star$)}]
\item \label{condition: orthogonal} $\partial_r$, $R_\alpha$, and $\xi$ are mutually orthogonal on $U$ and $\bdry_r$ and $R_\alpha$ have length $e^{r/2}$.
\end{enumerate}
\cb Let $g$ be the average of $\hat g$ under the group action $G$. Then $g$ is preserved by $G$ and \ref{condition: orthogonal} holds.

From $\omega$ and $g$, we obtain the canonical $\omega$-compatible almost complex structure $J$ on $M$ by the usual polar decomposition argument; see for example \cite[Proposition 12.3]{Si} and \cite[Proposition 2.50]{MS1}.
More precisely, we define $A:TM\to TM$ by $\omega(u,v)=g(Au,v)$ and the almost complex structure $J$ by $J=(\sqrt{A^*A})^{-1}A$, where $A^*$ is the $g$-adjoint of $A$. It is not hard to check that $J$ is $\omega$-compatible and $G$-invariant and that $J$ maps $\partial _r \mapsto R_\alpha$ and $J\xi =\xi$. Hence \ref{condition: J} is satisfied if $\bdry M\not=\varnothing$.
\end{proof}

\begin{lemma}\label{lemma: family J}
Given almost complex structures $J^0$ and $J^1$ that are $\omega$-compatible, $G$-invariant, and satisfy \ref{condition: J} if $\bdry M\not=\varnothing$, there exists a $1$-parameter family of almost complex structures $\{J^\tau\}_{\tau \in [0,1]}$ connecting $J^0$ and $J^1$ such that for each $\tau\in[0,1]$, $J^\tau$ is $\omega$-compatible, $G$-invariant, and satisfies \ref{condition: J} if $\bdry M\not=\varnothing$.
\end{lemma}

\begin{proof}
Define the metrics $g^i (\cdot, \cdot) := \omega(\cdot, J^i \cdot)$ for $i\in \{0,1\}$.
We can connect $g^0$ and $g^1$ by a $1$-parameter family of $G$-invariant metrics $\{g^\tau\}_{\tau\in [0,1]}$. It is not hard to see that we can take the $g^\tau$ so that \ref{condition: orthogonal} holds for each $\tau\in[0,1]$. Then we can define $\{J^\tau\}_{\tau \in [0,1]}$ as in the proof of Lemma~\ref{lemma: J}.
\end{proof}

From now on we assume $\omega$, $J$, and $g$ are compatible and $G$-invariant. \cbu
It is easy to check that we can further choose $J$ such that for any $p\in L_0 \cap L_1$, $J (T_p L_0) = T_p L_1$.
In later calculations, we implicitly use an identification of $(T_pM, J)$ with $(\R^n \oplus i\R^n,i)$ that maps $T_p L_0$ to the $\R^n$ factor and $T_p L_1$ to the $i\R^n$ factor.
\cb

\subsection{Fredholm setup}

Let $S=\R\times [0,1]$ with coordinates $(s,t)$ and the standard complex structure $j$ which maps $\bdry_s\mapsto \bdry_t$.  Let $p,q\in L_0\cap L_1$.\footnote{We will be using $p$ for both a point in $L_0\cap L_1$ and the $L^p$-exponent. Hopefully this will not create any confusion.}

For $k\geq 2$, let $\mathcal B^{k+1,p}=\mathcal B^{k+1,p}(p,q;A)$ be the space of maps $u:S\to M$ in $W^{k+1,p}(S, M)$ satisfying \hyperlink{A2}{(A2)}--\hyperlink{A4}{(A4)} and such that there exist $\rho_+,\rho_-\in \R$, $\xi_+\in W^{k+1,p}(S,T_pM)$, and $\xi_-\in W^{k+1,p}(S,T_qM)$ for which
\begin{itemize}
\item $u(s,t)=\exp_p\xi_+(s,t)$ for $s\geq\rho_+$,
\item $u(s,t)=\exp_q \xi_-(s,t)$ for $s\leq \rho_-$.
\end{itemize}
Here the exponential map $\exp$ is taken with respect to the $G$-invariant $g$.  Let
$$\pi:\mathcal E^{k,p}=\mathcal{E}^{k,p}(p,q;A)\to\mathcal B^{k+1,p}$$
be the smooth Banach bundle with fiber
$$\mathcal E^{k,p}_u=\pi^{-1}(u)=W^{k,p}(S, \wedge ^{0,1} S\otimes_J u^{*}TM).$$
Then
\begin{gather*}
\overline{\partial}_J: \mathcal B^{k+1,p} \to \mathcal E^{k,p}, \quad u\mapsto (u_s+J(u)u_t)\otimes_J (ds-idt)
\end{gather*}
is a Fredholm section and $\overline{\partial}_J^{-1}(0)=\widetilde{\mathcal M}_J(p,q)$.

Let $\nabla$ be the Levi-Civita connection on $M$ with respect to $g$. Let $D_u$ be the differential
$$(\overline \partial_J)_*: T_u \mathcal B^{k+1,p} \to T_{(u,\overline{\partial}_Ju)} \mathcal E^{k,p}$$
postcomposed with the projection to $\mathcal E^{k,p}_{(u,\overline \partial_J u)}$.  Let us write $\mathcal{W}^{k+1,p} (S, u^* TM)$ for $\xi\in W^{k+1,p} (S, u^* TM)$ satisfying $\xi(s,0)\in T_{u(s,0)}L_0$ and $\xi(s,1)\in T_{u(s,1)}L_1$. Then, by \cite[Proposition~3.1.1]{MS2},
$$D_u: \mathcal{W}^{k+1,p} (S, u^* TM)\to W^{k,p} (S, \wedge^{0,1}S\otimes _J u^* TM)$$
is given by
\begin{align} \label{formula: D}
D_u \xi & = \tfrac{1}{2}(\nabla \xi +J\nabla \xi \circ j) -\tfrac{1}{2}J(\nabla_\xi J)\partial_Ju \\
\nonumber & =  \tfrac{1}{2}\left[ (\nabla_s \xi+J\nabla_t\xi)-\tfrac{1}{2}J(\nabla_{\xi} J)(u_s-Ju_t) \right]\otimes_J(ds-idt).
\end{align}
By abuse of notation, we are not distinguishing between sections of $u^*TM$ and sections of $TM$ along $u$.

In what follows we will usually write $\pi:\mathcal{E}\to\mathcal{B}$. Note that, as $s\to \pm \infty$, $(\nabla_s \xi+J\nabla_t\xi)\to \bdry_s \xi + J(p) \bdry_t\xi$ and $u_s,u_t\to 0$.  This motivates the following definition.

\subsection{The asymptotic operator}

Consider $$\mathcal W_p=\{ \xi \in C^{\infty}([0,1],T_pM) ~|~ \xi(i)\in T_pL_i, i=0,1\},$$
with inner product
$$\langle \xi_1, \xi_2 \rangle=\int_0^1 g_p( \xi_1(t),\xi_2(t)) dt.$$
The {\em asymptotic operator} $A=A_p:\mathcal W_p\to \mathcal W_p$ is the self-adjoint operator
$$A\xi(t)= -J(p)\tfrac{\partial }{\partial t}\xi(t).$$

We list the eigenvalues of $A$
$$\cdots \leq \lambda_{-2} \leq \lambda_{-1}<0< \lambda_1\leq \lambda_2 \leq \cdots$$
with corresponding eigenfunctions
$$\cdots, f^p_{-2}, f^p_{-1}, f^p_1,f^p_2 \cdots,$$
chosen so that the $f_i^p$ form an $L^2$-orthonormal basis of $\mathcal W_p$.

\s\n
{\em Model calculation for the adjoint.}  Consider the map $u:\R\times[0,1]\to \C$ with boundary conditions $u(s,0)\in L_0=\R$ and $u(s,1)\in L_1=i\R=J\R$ and decay conditions $\lim_{s\to \pm \infty} u(s,t) =0$. Consider the Cauchy-Riemann operator $Du={\frac{\bdry u}{\bdry s}} + J{\frac{\bdry u}{\bdry t}}$.

We calculate the adjoint operator $D^*v$,
\cbu 
for any compactly supported $v:\R\times[0,1]\to \C$.
\cb
 It satisfies $\langle Du,v\rangle = \langle u, D^* v\rangle$, where $\langle,\rangle$ denotes the $L^2$-norm.  More precisely, we have
\begin{align}
\nonumber \int_{\R\times[0,1]} (\tfrac{\bdry u}{\bdry s} + J\tfrac{\bdry u}{\bdry t}, & v)  ds dt  = \int_{\R\times[0,1]}(\tfrac{\bdry u}{\bdry s},v)+ (J\tfrac{\bdry u}{\bdry t}, v) ds dt\\
\nonumber & = \int_{\R\times[0,1]}  \left(\tfrac{\bdry}{\bdry s}(u,v) - (u,\tfrac{\bdry v}{\bdry s}) +\tfrac{\bdry}{\bdry t}(Ju,v) - (Ju, \tfrac{\bdry v}{\bdry t}) \right)ds dt\\
\label{eqn: adjoint} &= - \int_{\R\times[0,1]}  (u, \tfrac{\bdry v}{\bdry s}-J \tfrac{\bdry v}{\bdry t})dsdt +\int_{\R\times[0,1]}  \tfrac{\bdry}{\bdry t}(Ju,v) ds dt
\end{align}
Here $(\cdot,\cdot)$ is the real part of the standard Hermitian inner product on $\C$. Observe that:
$$\int_{-\infty}^\infty \tfrac{\bdry }{\bdry s}(u,v) ds = \left.(u,v)\right|_{s=-\infty}^{s=+\infty}=0$$
by the decay conditions at $s=\pm \infty$.  We also have
\begin{equation} \label{octopus}
\int_0^1\tfrac{\bdry}{\bdry t} (Ju,v)dt=\left.(Ju,v)\right|_{t=0}^{t=1}= \left.\omega(Ju,Jv)\right|_{t=0}^{t=1}  = \left.\omega(u,v)\right|_{t=0}^{t=1}.
\end{equation}

The following claim implies the adjoint is $D^*v= -({\frac{\bdry v}{\bdry s}}-J{\frac{\bdry v}{\bdry t}})$, subject to the restriction of the domain to $v$ satisfying $v(s,0)\in L_0=\R$ and $v(s,1)\in L_1=J\R$.

\begin{claim}
If $\langle Du,v\rangle =0$ for all $u$, then $v$ satisfies $D^*v=0$ and boundary conditions $v(s,0)\in L_0=\R$ and $v(s,1)\in L_1=J\R$.
\end{claim}

\begin{proof}
By Equations~\eqref{eqn: adjoint} and ~\eqref{octopus}, if $\langle Du,v\rangle=0$ for all $u$, then
$$\int_{\R\times[0,1]}  (u, D^*v)dsdt +\int_{\R}\left((Ju(s,1),v(s,1))-(Ju(s,0),v(s,0))\right)ds =0$$
for all $u$.  We can decouple this equation into two pieces by considering $u$ that are supported in the interior of $\R\times[0,1]$ and on small neighborhoods of boundary points. Hence we obtain the conditions $D^*v=0$ and $v(s,0)\in L_0=\R$ and $v(s,1)\in L_1=J\R$.
\end{proof}
\cb

\subsection{Interior semi-global Kuranishi charts} \label{charts}

Let us first consider a single moduli space
$$\mathcal M_J=\mathcal M_J(p,q;A)=\widetilde{\mathcal{M}}_J(p,q;A)/\R.$$
We will often suppress the almost complex structure $J$ from the notation when it is clear from the context.
Let us also abbreviate $\widetilde{\mathcal{M}}=\widetilde{\mathcal{M}}(p,q;A)$, $\mathcal{B}=\mathcal{B}(p,q;A)$, and $\mathcal{E}=\mathcal{E}(p,q;A)$.


\begin{defn}
An {\em interior semi-global Kuranishi chart} is a quadruple $(\K, \pi: \E\to \V, \overline\partial,\psi),$
where:
\begin{enumerate}
\item[(i)] $\K \subset \mathcal M$ is a large compact subset; if $\mathcal M$ is compact, we take $\K=\mathcal M$;
\item[(ii)] $\pi:\E\to \V$, called the {\em obstruction bundle}, is a finite rank vector bundle over a finite-dimensional manifold;
\item[(iii)] $\overline{\partial}: \V\to \E$ is a section;
\item[(iv)] $\psi: \overline{\partial}^{-1}(0)\to \mathcal M$ is a homeomorphism onto an open subset of $\mathcal{M}$ and $\K\subset \op{Im}(\psi)$;
\item[(v)] $\dim \V-\op{rk}\E=\op{vdim}\mathcal{M}$.
\end{enumerate}
If the group $G$ \cbu acts on \cb $(\K, \pi: \E\to \V, \overline\partial,\psi)$, then the Kuranishi chart is {\em $G$-invariant}.

A section ${\mathfrak s}$ of $\pi: \E\to \V$ that is transverse to $\overline \partial $ is an {\em obstruction section}.
\end{defn}

\begin{notation}
In (iii) we are abusing notation and writing $\overline \bdry$ for the section to indicate that it descends from $\overline\bdry:\mathcal{B}\to \mathcal{E}$; for the charts we construct, the sections $\overline\bdry$ are consistent with one another.  We will also often abuse notation and write $\K\subset \V$ without referring to the map $\psi$.
\end{notation}

The goal of this subsection is to construct a $G$-equivariant interior semi-global Kuranishi chart over a large $G$-invariant compact subset $\K\subset \mathcal{M}$.  


Let $B_q\subset M$ be a sufficiently small disk neighborhood of $q\in L_0\cap L_1$.  Given $m\in B_q$, let
$$\Gamma^q_m: T_qB_q\to T_mB_q$$
be the parallel transport with respect to the Levi-Civita connection of $g$ along the shortest geodesic from $q$ to $m$.  Next we define the $t\in [0,1]$-dependent section $F^q_j:[0,1]\times B_q\to TB_q$ of $TB_q\to B_q$ by
$$F^q_j(t,m)=\Gamma^q_m (f^q_j(t)),$$
where $f^q_j$ are the eigenfunctions of $A_q$.

\cbu
\begin{defn}[The map $a_q$] \label{location epsilon}
Let $\mathcal{P}(B_q)$ be the space of $C^1$-paths $\gamma:[0,1]\to B_q$ satisfying $\gamma(i)\in L_i$ for $i\in\{0,1\}$. We then define a map
$$a_q:\mathcal{P}(B_q)\to \R$$
as follows: Let $v_\gamma:(-\infty,0]\times [0,1]\to M$ be a $C^1$-map such that $v_\gamma(0,t)=\gamma(t)$, $v_\gamma(s,i)\in L_i$ for $i\in\{0,1\}$, and $\lim_{s\to -\infty} v_\gamma(s,t)=q$. Then $v_\gamma$ is a path in $\mathcal{P}(B_q)$ from the constant path at $q$ to $\gamma$.  Then let
\begin{equation}
a_q(\gamma)=\int_{(-\infty,0]\times[0,1]} v_\gamma^* \omega.
\end{equation}
Note that $a_q(\gamma)$ does not depend on the choice of path $v_\gamma$.
\end{defn}

By the monotonicity lemma, there exists $\varepsilon>0$ such that for each nonconstant $v\in \widetilde{\mathcal M}$, there exists a unique value $s^q_{v,\varepsilon}$ of $s\in \R$ which satisfies the following:
\begin{enumerate}[label=\textbf{(sq)}]
\item \label{def: sq} the path $\gamma_{v,s}(t)= v(s,t)$ is contained in $B_q$ and $a_q(\gamma_{v,s})=\varepsilon$.
\end{enumerate}
Note that if $v'(s,t) = v(s + s_0, t)$ then $s^q_{v',\varepsilon} = s^q_{v,\varepsilon} - s_0$. 

\cb
We can also define $s^q_{u,\varepsilon}$ for $u\in \mathcal{B}$ which is $C^1$-close to $v$.

Let $\mathcal U=\mathcal U_v\subset\mathcal{B}$ be a sufficiently small open neighborhood of \cbu $v\in \widetilde{\mathcal M}$\cb.  Fix $\delta>0$ small.  \cbu
We pick a smooth bump function $\beta: \R \to [0,1]$ such that 
\begin{enumerate}
\item[(a)] $\beta(s)=1$ for $s\in [-1,1]$, and
\item[(b)]  $\beta(s)=0$ for $s\not\in [-2,2]$.
\end{enumerate}
\cb We construct a section $\tilde f_j^q=\tilde f_j^{q,\delta}$ of $\mathcal{E}|_{\mathcal{U}} \to \mathcal U$ as follows.
For each $u\in \mathcal U$, we define
\begin{equation}
\tilde f^q_j(u)=\beta_u^q \cdot u^* F^q_j \otimes_{\C} (ds-idt)\in W^{k,p}(S, \Lambda ^{0,1}S\otimes u^*{TM}),
\end{equation}
where $\beta_u^q:\R\to [0,1]$ is a smooth bump function of $s$ \cbu defined by 
$$\beta_u^q(s) = \beta(\delta^{-1}(s - s^q_{u,\varepsilon})).$$

\begin{center}
\begin{figure}
\begin{tikzpicture}[scale = 0.8]
\filldraw[black!10!white]
(-1,-2.5) rectangle (0,2.5);
\draw [dotted] (-1,2.5) -- (0, 2.5);
\draw (-1,-2.5) -- (-1,2.5);
\draw (0,-2.5) -- (0,2.5);
\draw[dotted] (-1,-2.5) -- (0,-2.5);
\node [left] at (-1,-1) {$s^q_{u,\varepsilon}$};
\draw[dashed] (-1,-1) -- (0,-1);
\draw (1,2.5) -- (1, -0.5);
\draw (1, -1.5) -- (1, -2.5);
\draw (1, -0.5) .. controls (1, -0.7) and (1.5, -0.7) .. (1.5, -1);
\draw (1, -1.5) .. controls (1, -1.3) and (1.5, -1.3) .. (1.5, -1);
\node[right] at (1, -2) {$\beta^q_{u}$};
\node [below] at (-0.4,-2.5) {$-\infty$};
\node[above] at (-0.4, 2.5) {$+\infty$};
\draw[->] (2,1) -- (3,1);
\node[above] at (2.5, 1) {$u$};
\draw (5.1, 2.8) .. controls (4, 0) .. (5.1, -2.8);
\draw (4.9, 2.8) .. controls (6, 0) .. (4.9, -2.8);
\node[left] at (4.5, -1) {$L_1$};
\node[right] at (5.5, 1) {$L_0$};
\node[right] at (5,2.5) {$p$};
\node[right] at (5,-2.5) {$q$};
\end{tikzpicture}
\caption{}
\end{figure}
\end{center}

\cb We denote by $E^{\ell}=E^{q,\ell} \to \mathcal U$ the vector subbundle of $\mathcal E|_{\mathcal U}$ spanned by the sections \cbu $\tilde f_{-1}^q,\dots,\tilde f_{-\ell}^q$. 
The $\R$-translation of $S=\R\times [0,1]$ induces an $\R$-action on $\mathcal E \to \mathcal B$ with respect to which the sections $\overline \partial$ and $\tilde f_i^q$ are equivariant. We denote by $\E^\ell \to \mathbb U$ the quotient of the bundle $E^{\ell} \to \mathcal U$ by the $\R$-action. 
We also introduce the vector space
\begin{equation} \label{eqn: e ell}
e^\ell=e^{q,\ell}=\R\langle f_{-1}^q,\dots, f_{-\ell}^q \rangle.
\end{equation}

\begin{prop} \label{prop: large chart}
There exist a sufficiently large $\ell$ and a sufficiently small open neighborhood $\mathcal N(\K)\subseteq \mathcal B/\R$ of $\K$ such that the vector bundle $\mathbb E^{\ell}\to \mathcal N(\K)$, obtained by patching together charts of the form $\mathbb E^\ell\to \mathbb{U}$ with $\mathbb U = \mathbb U_{[v]} := \mathcal U_v /\R$ and  $[v]\in \K$, is transverse to the section $\overline\partial|_{\mathcal N(\K)}$ and is trivial with fibers that are canonically identified with $e^\ell$.
\end{prop}

The proof is similar to that of \cite[Theorem~5.1.2]{BH2} and will be omitted.  In a nutshell, this is because $D^*_u (\zeta \otimes (ds - idt))$ is approximated by $-(\partial_s \zeta + A \zeta)$ for $s\ll 0$ and each nonzero element of $\op{Ker}D^*_u$ has a negative end that is dominated by $e^{-\lambda_j s} f_j(t)$ for some $j<0$.

\cb
We then define
$$\V:=\overline{\partial}^{-1}(\E^\ell|_{\mathcal N( \K)})\subset \mathcal N(\K)$$
and restrict $\mathbb E^{\ell}\to \mathcal N(\K)$ to $\V$.   By shrinking $\V$ if necessary, we may assume that $\V$ is $G$-equivariant.
This completes the construction of a $G$-equivariant interior semi-global Kuranishi chart for $\K$.

In view of the identification of the fibers of $\E^\ell$ with $e^\ell$, we will usually take an obstruction section ${\mathfrak s}$ on $\E^{\ell}\to \V$ to be a generic point in $e^\ell=e^{q,\ell}$ which is sufficiently close to the origin.  A more specific choice of the generic point $s_q\in e^{q,\ell}$ will be made in Section~\ref{subsection: curve counting}.

\subsection{Boundary semi-global Kuranishi charts}\label{charts'}

In this subsection, we explain how to construct Kuranishi charts for curves that are close to breaking.

\subsubsection{Simplest case} \label{subsubsection: simplest case}

Let us consider the simplest situation where
$$\mathcal M_1=\mathcal M(p,r;A_1),\quad \mathcal M_2=\mathcal M(r,q;A_2),\quad \mathcal M_3= \mathcal M(p,q; A_1+A_2),$$
$\bdry \mathcal{M}_3=\mathcal{M}_1\times \mathcal{M}_2$, and $\mathcal{M}_3$ is $G$-invariant.   Let $\K_1$, $\K_2$, $\K_3$ be compact subsets of $\mathcal M_1$, $\mathcal M_2$, $\mathcal M_3$, respectively,
$$(\K_1,  \E_1\to \V_1, \overline\partial),\quad (\K_2, \E_2 \to \V_2,\overline\partial),\quad (\K_3, \E_3\to \V_3,\overline\partial)$$
be the corresponding interior $G$-equivariant semi-global Kuranishi charts, and ${\mathfrak s}_1\in e^{r,\ell}$, ${\mathfrak s}_2\in e^{q,\ell}$, ${\mathfrak s}_3\in e^{q,\ell}$ be the obstruction sections.

We will construct a boundary semi-global Kuranishi chart $\E_{(12)}\to \V_{(12)}$ over the curves of $\mathcal{M}_3$ that are close to breaking.  Let $\sigma>0$ be small. \cbu

\begin{defn}[Close to breaking] \label{defn: close to breaking}
An element $u\in \mathcal B(p,q;A_1+A_2)$ (resp.\ $[u]\in  \mathcal B(p,q;A_1+A_2)/\R$)  is {\em $\sigma$-close} to a broken strip $([u_1],[u_2])\in \V_1 \times \V_2 $ if there exist representatives $u_1$, $u_2$ of $[u_1]$, $[u_2]$ (resp. $u$, $u_1$, $u_2$ of $[u]$, $[u_1]$, $[u_2]$) such that
\begin{itemize}
\item  $u|_{[\sigma^{-1},\infty)\times [0,1]}$ is $\sigma$-close in the $C^1$-norm to $u_1|_{[\sigma^{-1},\infty)\times[0,1]}$,
\item  $u|_{(-\infty,-\sigma^{-1}]\times [0,1]}$ is $\sigma$-close in the $C^1$-norm to $u_2|_{(-\infty,-\sigma^{-1}]\times[0,1]}$, 
\item $u|_{[-\sigma^{-1},\sigma^{-1}]\times [0,1]}$, $u_1|_{(-\infty,\sigma^{-1}]\times[0,1]}$, and $u_2|_{[-\sigma^{-1},\infty)\times[0,1]}$ are $\sigma$-close in the $C^1$-norm to the constant map to the point $r$.
\cb
\end{itemize}
\end{defn}

\cb
Let $\widetilde {\mathcal G}_{\sigma}(\V_1,\V_2)\subset  \mathcal B(p,q;A_1+A_2)$ be the subset of maps $u$ that are $\sigma$-close to some broken strip $([u_1],[u_2])\in \V_1 \times \V_2$ and let $\mathcal G_\sigma (\V_1,\V_2):=\widetilde {\mathcal G}_{\sigma}(\V_1,\V_2)/\R$.

For each $u\in \widetilde {\mathcal G}_\sigma (\V_1,\V_2)$, there exists a unique value  $s^r_{u,\varepsilon}$ of $s\in \R$ which satisfies the following:
\begin{enumerate}[label = \textbf{(sr)}]
\item \label{condition: sr} the path $\gamma_{u,s}(t)= u(s,t)$ is contained in $B_r$ and $a_r(\gamma_{u,s})=\varepsilon$. 
\end{enumerate}
Then for each $u\in \widetilde {\mathcal G}_\sigma (\V_1,\V_2)$ we define
$$\tilde f^{r}_j(u)=\beta^r_u \cdot u^* F^r_j \otimes_{\C} (ds-idt)\in W^{k,p}(S, \Lambda ^{0,1}S\otimes u^*{TM}),$$
\cbu where $\beta^r_u:\R\to[0,1]$ is the smooth bump function
$\beta^r_u(s)=\beta(\delta^{-1}(s-s^r_{u,\varepsilon}))$,
and $\beta$ is as before.\cb 

Let $E^{\ell}(\V_1,\V_2) \to \widetilde {\mathcal G}_{\sigma}(\V_1,\V_2)$ be the vector subbundle of $\mathcal E|_{\widetilde {\mathcal G}_{\sigma}(\V_1,\V_2)}$ spanned by the sections \cbu $\tilde f_{-1}^{r},\dots,\tilde f_{-\ell}^{r}$ and $\tilde f_{-1}^q,\dots,\tilde f_{-\ell}^q$. \cb  By linear gluing (a simpler version of Theorem~\ref{thm: gluing} described below) for $\sigma>0$ sufficiently small, $E^{\ell}(\V_1,\V_2) \to \widetilde {\mathcal G}_{\sigma}(\V_1,\V_2)$ is transverse to $\overline\bdry$. We then define
$$V_{(1,2)}:=\overline{\partial}^{-1}(E^\ell(\V_1,\V_2) )\subset \widetilde {\mathcal G}_{\sigma}(\V_1,\V_2).$$
The quotient of $E^\ell(\V_1,\V_2)|_{V_{(1,2)}}\to V_{(1,2)}$ by the $\R$-translation is denoted by:
$$\pi_{(1,2)}:\E^{\ell}_{(1,2)} \to \V_{(1,2)}.$$
Observe that $\E^{\ell}_{(1,2)}$ is a trivial vector bundle whose fibers are canonically identified with $e^{r,\ell}\oplus e^{q,\ell}$.

Let us fix $\varepsilon'$ satisfying $0<\varepsilon'\ll \varepsilon$. Suppose $\sigma=\sigma(\varepsilon')>0$ is sufficiently small.

\begin{defn}[Neck length]
The {\em neck length function} is the function
\begin{gather}
\nonumber \mathfrak{nl}: \widetilde {\mathcal G}_{\sigma}(\V_1,\V_2)\to \R^+, \\
\label{eqn: defn of neck length} u\mapsto   s^r_{u,-\varepsilon'}-s^r_{u,\varepsilon'},
\end{gather}
where $s^r_{u, -\varepsilon'}$ and $s^r_{u,\varepsilon'}$ are the unique values defined as in (sr) above.
\end{defn}

Observe that $\mathfrak{nl}: \widetilde {\mathcal G}_{\sigma}(\V_1,\V_2)\to \R^+$ descends to $\mathfrak{nl}: \mathcal{G}_\sigma (\V_1,\V_2)\to \R^+$.
Pick \cbu $\mathcal{L}=\mathcal{L}(\varepsilon',\sigma)>0$ large \cb and $\varepsilon''>0$ small.  After some modifications we may assume that:
\begin{enumerate}[label= (C)]
\item $\V_3$ and $\V_{(1,2)}$ cover $\mathcal{M}_3$;
\end{enumerate}
\begin{enumerate}[label=(C$_3$)]
\item $\V_3\cap \mathcal{M}_3$ consists of $[u] \in \mathcal{M}_3 - {\mathcal G}_{\sigma}(\V_1,\V_2)$ and $[u]\in {\mathcal G}_{\sigma}(\V_1,\V_2)\cap\mathcal{M}_3$ satisfying $\mathfrak{nl}([u]) <\mathcal{L}$;
\end{enumerate}
\begin{enumerate}[label = (C$_{(1,2)}$)]
\item $\V_{(1,2)}\cap\mathcal{M}_3$ consists of $[u] \in {\mathcal G}_{\sigma}(\V_1,\V_2)\cap\mathcal{M}_3$ satisfying $\mathfrak{nl}([u])> \mathcal{L}-\varepsilon''$;
\end{enumerate}
\begin{enumerate}[label = (G)]
\item $G$ acts equivariantly on $\E_{(1,2)}\to \V_{(1,2)}$. 
\end{enumerate}

The bundles $\E_3\to \V_3$ and $\E_{(1,2)}\to \V_{(1,2)}$ are related by  the restriction-inclusion morphism:  we first restrict $\E_3\to \V_3$ to
$$\V_{3,(1,2)}:=\V_3\cap \{\mathcal{L}-\varepsilon''< \mathfrak{nl}([u])<\mathcal{L}\}$$
and take the natural inclusion into $\E_{(1,2)}\to \V_{(1,2)}$, recalling that the fibers of $\E_3$ are canonically identified with $e^{q,\ell}$ and the fibers of $\E_{(1,2)}$ are canonically identified with $e^{r,\ell}\oplus e^{q,\ell}$.

\begin{defn}[The function $\zeta$] \label{defn: zeta}
Choose $0<\varepsilon'''\ll \varepsilon''$. Let
$$\zeta:[0,\infty) \to [0,1]$$
be a smooth function such that
\begin{itemize}
\item $\zeta([0,\mathcal{L}+\varepsilon'''])= 0$,
\item $\zeta([\mathcal{L}+\varepsilon''-\varepsilon''',\infty))=1$, and
\item its restriction to $(\mathcal{L}+\varepsilon''',\mathcal{L}+\varepsilon''-\varepsilon''')$ is a diffeomorphism onto $(0,1)$.
\end{itemize}
\end{defn}

We then set
$${\mathfrak s}_{(1,2)}=(\zeta(\mathfrak{nl}) s_r,s_q)\in  e^{r,\ell}\oplus e^{q,\ell}.$$
In particular, $\mathfrak{s}_{(1,2)}=(s_r,s_q)$ on $\mathfrak{nl}\geq \mathcal{L}+\varepsilon''$ and $\mathfrak{s}_{(1,2)}=(0,s_q)$ on $\mathfrak{nl}\leq \mathcal{L}$. By the restriction-inclusion, ${\mathfrak s}_{(1,2)}$ is consistent with $\mathfrak{s}_3$.

\cbu
\subsubsection{Order of choice of constants}

We outline the 
order in which the auxiliary constants are chosen.

\begin{enumerate}
\item Choose $\varepsilon_1 > 0$ small such that $s_{u,\varepsilon'}^r$ is defined for any $0 < \varepsilon' < \varepsilon_1$ and any $u \in \widetilde{\mathcal{M}} (p,r; A_1)$.  Then choose $\ell_1$ and $\V_1$.

\item Choose $\varepsilon_2 > 0$ small such that $s_{v,\varepsilon'}^q$ and $s_{w,\varepsilon'}^q$ are defined for any $0 < \varepsilon' < \varepsilon_2$ and any $v \in \widetilde{\mathcal M}(r,q;A_2)$ and $w \in \widetilde{\mathcal M}(p,q;A_1 + A_2)$. Then choose $\ell_2$ and $\V_2$. 


\item \label{indep of ell} Choose $\varepsilon'>0$ such that $0<\varepsilon'\ll \varepsilon_1$ and then $\sigma > 0$ small such that
$s_{w,\varepsilon'}^r$ and $s_{w,-\varepsilon'}^r$ are defined for any $w \in \widetilde {\mathcal G}_{\sigma}(\V_1,\V_2)$. Here $\sigma$ can be chosen to be independent of $\ell_1$ and $\ell_2$, but we may need to shrink $\V_i$ satisfying $\K_i \subseteq  \V_i$ for $i = 0, 1$. This is because if $w$ is close to breaking into $\V_1 \times \V_2$ and $\V_1\times \V_2$ is a sufficiently small neighborhood of $\K_1\times \K_2$, then $w$ is close to breaking into $\K_1 \times \K_2$, which is $\ell_1, \ell_2$-independent. The neck length $\mathfrak{nl}(w)$ is then given by $s^r_{w,-\varepsilon'}-s^r_{w,\varepsilon'}$. 

\item Define $\V_{(1,2)}\subset {\mathcal G}_{\sigma}(\V_1,\V_2)$.

\item Choose a compact $\K_3 \subseteq \mathcal M(p,q;A_1 + A_2)$ such that if $[w] \in \mathcal M(p,q;A_1 + A_2) - \K_3$, then $[w]$ is $\sigma$-close to breaking into $\V_1 \times \V_2$, i.e., $[w] \in \mathcal{G}_\sigma (\V_1,\V_2)$.
\item Pick $\mathcal L>0$ large and $\varepsilon''>0$ small, and enlarge $\K_3$ if necessary so that $$\K_3 \cap \V_{(1,2)} \supseteq \{ [w] \in \mathcal M(p,q;A_1 + A_2) ~|~ \mathcal L - \varepsilon'' < \mathfrak{nl}([w]) < \mathcal L\}.$$
\item Choose $\ell \in \N$ such that $\E^\ell|_{\K_3}$ is transverse to $\overline \partial$.
\item With the choice of $\ell$, we may need to increase $\ell_2$ so that $\ell = \ell_2$. By (\ref{indep of ell}) this update does not affect $\sigma$, $\mathcal{L}$, $\varepsilon''$. To reduce the number of constants, we also choose to update $\ell_1$ so that $\ell_1 = \ell_2 = \ell$.
\item Trim $\V_3$ and $\V_{(1,2)}$ so that (C$_3$) and (C$_{(1,2)}$) are satisfied.
\item The constant $0 < \varepsilon''' \ll \varepsilon''$ is as defined in Definition~\ref{defn: zeta}. Then we interpolate between the sections $\mathfrak s_3$ and $\mathfrak s_{(1,2)}$.
\end{enumerate}
\cb

\subsubsection{General case}
\cbu
In general, we construct boundary semi-global Kuranishi charts by induction on the energy.

\s\n
{\em Ordering the moduli spaces.} We will explain how to order the moduli spaces $\mathcal M (p,q;A)$. We remark that even if  $\mathcal M(p,q; A) = \emptyset$, we still need to construct a Kuranishi chart for $\mathcal M(p,q; A)$ (i.e., include $\mathcal M(p,q;A)$ in our list), if $A = A_1 + A_2$ for some $A_1 \in \pi_2(p,r)$ and $A_2 \in \pi_2(r,q)$, and $\mathcal M(p,r; A_1) \neq \emptyset$ and $\mathcal M(r,q; A_2) \neq \emptyset$.

We define 
$$\mathfrak{M} = \{ (p, q, A) ~ | ~  p,q \in L_0 \cap L_1 \text{ and } A\in \pi_2(p,q) \},$$
$$\mathfrak{M}^{1+} = \{ (p,q,A)\in \mathfrak{M} ~ | ~  \mathcal M(p,q; A) \neq \emptyset \},$$
$$\lambda_1 = \inf \{\omega(A) ~|~ (p,q; A)\in \mathfrak{M}^{1+} \},$$
and
$$\mathfrak{M}^1 =\{ (p,q; A)\in \mathfrak{M}^{1+} ~|~   \omega(A) = \lambda_1  \}.$$

Suppose that $\mathfrak{M}^k$ has been inductively defined for all $k < m$.  We then define $\mathfrak{M}^{m+}$ to be the set of 
$$(p,q,A) \in \mathfrak M \backslash \left( \cup_{i=1}^{m-1}  \mathfrak M^i \right) $$ 
such that either
\begin{enumerate}
\item $ \mathcal M(p,q; A) \neq \emptyset$; or 
\item There exist $r \in L_0 \cap L_1$, $A_1\in \pi_2(p,r)$ and $A_2 \in \pi_2(r,q)$ satisfying 
	\begin{enumerate} 
	\item[(2a)] $A = A_1 + A_2$,
	\item[(2b)] $(p,r; A_1) \in \mathfrak M^{k_1}$ for some $k_1 < m$, and 
	\item[(2c)] $(r,q;A_2) \in \mathfrak M^{k_2}$ for some $k_2 < m$.
	\end{enumerate}
\end{enumerate}
Let 
$$\lambda_m = \inf \{\omega(A) ~|~ (p,q,A) \in \mathfrak M^{m+}\},$$ and 
$$\mathfrak M^m = \{ (p,q,A)\in \mathfrak M^{m+} ~|~ \omega(A) = \lambda_m \}.$$
By Gromov compactness, one can see that for each $k$, $\mathfrak M^k$ is finite, and $\{\lambda_k ~|~ k \in \N \}\subset \R_{\geq 0}$ is nowhere dense.

We then order the elements of $\cup_{k=1}^\infty \mathfrak{M}^k$ as \begin{equation}\label{ordering of triples}
(p_1,q_1,A_1),(p_2,q_2,A_2),\dots
\end{equation}
so that it is consistent with the ordering 
$$\mathfrak{M}^1, \mathfrak{M}^2,\dots.$$
We denote $\mathcal M_i := \mathcal M(p_i, q_i; A_i)$.

\cb

\s

\cbu We choose an increasing sequence $N_1, N_2, \dots \to \infty$ of integers and for each $j$ we construct a semi-global Kuranishi structure $\mathcal{K}^{(j)}$ using $\mathcal{M}_1,\dots, \mathcal{M}_{N_j}$ and a section $\mathfrak{S}^{(j)}$ of $\mathcal{K}^{(j)}$. Later we will explain how to relate $\mathcal{K}^{(j)}$ and $\mathcal{K}^{(j+1)}$ and their sections.  

{\em For the moment we choose $N\gg 0$ and such that the finite set $\{\mathcal{M}_1,\dots, \mathcal{M}_N\}$ of moduli spaces is $G$-invariant.} \cb

Define the {\em source}, {\em target}, and {\em homotopy class maps}
$${\bf s}(\mathcal M(p,q;A))=p, \quad {\bf t}(\mathcal M(p,q;A))=q, \quad {\bf h}(\mathcal{M}(p,q;A))=A.$$

\begin{defn} \label{defn: index tuple}
A tuple $I=(i_1,\dots,i_k)$ with $i_j\in \{1,2,\dots,\rho\}$ is called an {\em index tuple} if $\omega({\bf h}(\mathcal M_{i_j}))>0$ for all $j$ and
${\bf t}(\mathcal M_{i_j})={\bf s}(\mathcal M_{i_{j+1}})$ for all $j<k$.
\end{defn}

If $k=1$, sometimes we write $i_1$ instead of $(i_1)$.

\begin{defn}
Let $I=(i_1,i_2,\dots, i_k)$ be an index tuple.
\be
\item An index tuple $I'$ is a {\em simple contraction} of $I$  if $I'$ is obtained by replacing a consecutive pair $i_j, i_{j+1}$ by $i'_j$ such that ${\bf h}(\mathcal{M}_{i'_j})={\bf h}(\mathcal{M}_{i_j})+{\bf h}(\mathcal{M}_{i_{j+1}}).$
\item An index tuple $I'$ is a {\em contraction} of $I$ if $I'$ is obtained from $I$ by a \cbu non-empty \cb sequence of simple contractions.  We write $I'<I$.
\item We write $c(I)$ for the index tuple $(i'_1)$ such that \cbu $(i'_1)\leq I$ (i.e., $(i'_1)< I$ or $(i'_1)=I$). \cb
\item Given $I'=(i'_1,\dots,i'_{k'})<I$, the {\em blocks of $I$ relative to $I'$} are groupings
$$(i_1,\dots,i_{l_1}),(i_{l_1+1},\dots, i_{l_2}),\dots, (i_{l_{k'-1}+1},\dots,i_{l_{k'}})$$
such that $c$ applied to the $j$th block yields $i'_j$. \cbu Note that the blocks are well-defined due to the requirement $\omega({\bf h}(\mathcal{M}_{i'_{j'}}))> 0$ for all $j' \in \{1, 2,\dots, k'\}$ \cb
\item Given $I'=(i'_1,\dots,i'_{k'})<I$, let
$$\delta(I,I')=\{i_1,\dots,i_{l_1-1},i_{l_1+1},\dots,i_{l_2-1},\dots,i_{l_{k'-1}+1},\dots,i_{l_{k'}-1}\},$$
where we are using block notation from (4).
\ee
\end{defn}

We can organize the set of index tuples as a category $\mathcal{I}$, called the {\em index tuple category}, with objects which are index tuples and a unique morphism from $I'$ to $I$ if $I'\leq I$.

Let $\K_i\subseteq \mathcal M_i$ be the large compact subsets over which we construct the equivariant interior semi-global Kuranishi chart $\mathcal C_i=(\K_i, \E_i \to \V_i,\overline\partial_i, \psi_i)$ and the obstruction section ${\mathfrak s}_i$.

Let $I=(i_1,\dots,i_k)$. The following construction of the boundary chart
$$(\pi_I:\E_I\to \V_I,\overline\bdry_I,\psi_I)$$
is a straightforward generalization of Section~\ref{subsubsection: simplest case}:
Let $\widetilde {\mathcal G}_{\sigma}(\V_{i_1},\dots,\V_{i_k})$ be the set of maps $u$ that are $\sigma$-close to a broken strip in $\V_{i_1}\times\dots\times \V_{i_k}$, defined in a manner analogous to Definition~\ref{defn: close to breaking}.  For convenience we will also write $\widetilde{\mathcal G}_{\sigma}(\V_i)$ for the set of maps $u$ that are $\sigma$-close to a map in $\V_i$.  Again we take \cbu $\mathcal{L}=\mathcal{L}(\varepsilon',\sigma)$ \cb and $\varepsilon''>0$.

\begin{defn}[Neck length] \label{defn: nl two}  Let $u\in \cup_{c(i_1,\dots,i_k)=(i)} \widetilde {\mathcal G}_{\sigma}(\V_{i_1},\dots,\V_{i_k})$.
\begin{enumerate}
\item The {\em neck length function} satisfies
$$\mathfrak{nl}'_{(i_1',i_2')}(u)=s^{r}_{u,-\varepsilon'}-s^{r}_{u,\varepsilon'}$$
if $u\in \widetilde {\mathcal G}_{\sigma}(\V_{i_1},\dots,\V_{i_k})$, $(i)<(i_1',i_2')< (i_1,\dots,i_k)$, and $r= {\bf t}(\mathcal{M}_{i_1'})={\bf s}(\mathcal{M}_{i_2'})$.
\item The {\em modified neck length function} satisfies
$$\mathfrak{nl}_{(i_1',i_2')}(u) =\left \{ \begin{array}{ll}
\lambda(\mathfrak{nl}'_{(i_1',i_2')}(u)) & \mbox{ if $\mathfrak{nl}'_{(i_1',i_2')}(u)$ is defined;}\\
0 & \mbox{ otherwise,} \end{array} \right.$$
where $\lambda: \R^+\to \R^{\geq 0}$ is \cbu a smooth function such that \cb $\lambda(x)=x$ for $x\geq \mathcal{L}-\varepsilon''$, $\lambda(x)=0$ for $x\leq \mathcal{L}-2\varepsilon''$, and $\lambda'(x)>0$ on $(\mathcal{L}-2\varepsilon'', \mathcal{L}-\varepsilon'')$.  We also write $\mathfrak{nl}_{i_j}(u)=\mathfrak{nl}_{(i_1',i_2')}(u)$ if $c(i_1,\dots,i_j)=i_1'$.
\end{enumerate}
\end{defn}

We then define the boundary charts $\pi_{ I}:\E_I\to \V_I$, $I=(i_1,\dots, i_k)$, whose fibers are canonically identified with $e^{r_1,\ell}\oplus\dots\oplus e^{r_{k},\ell}$ and such that:
\be
\item[(C$_I$)] $\V_I\cap\mathcal{M}_{c(I)}$ consists of $[u] \in {\mathcal G}_{\sigma}(\V_{i_1},\dots, \V_{i_k})\cap\mathcal{M}_{c(I)}$ satisfying
\be
\item $\mathfrak{nl}_{i_j}([u])> \mathcal{L}-\varepsilon''$ for all $j<k$ and
\item $\mathfrak{nl}_{i''_j}([u])< \mathcal{L}$  for all $i''_j\in \delta(I'',I)$ where $I< I''=(i''_1,\dots,i''_{k''})$.
\ee
\ee
The section $\overline\bdry_I$ is the $\overline\bdry$-operator restricted to $\V_I$ and $\psi_I: \overline\bdry^{-1}_I(0)\to \V_I$ is the obvious inclusion.

Observe that $G$ acts on the set of index tuples.  Let $G_I\subset G$ be the stabilizer of $I$.  By trimming $\V_I$ if necessary, we may assume that $G_I$ acts on $\E_I\to\V_I$.

Next we discuss the restriction-inclusion morphism
$$\phi_{I',I}:(\pi_{I'}:\E_{I'}\to \V_{I'})\to (\pi_I:\E_I\to \V_I),$$
where $I'=(i'_1,\dots,i'_{k'})<I=(i_1,\dots,i_k)$. We first restrict $\E_{I'}\to \V_{I'}$ to
$$\E_{I',I}:=\E_{I'}|_{\V_{I',I}}\to \V_{I',I}:=\V_{I'}\cap \{\mathcal{L}-\varepsilon''< \mathfrak{nl}_{i_j}([u])<\mathcal{L}, \forall i_j\in \delta(I,I')\}.$$
We then consider the inclusion of vector bundles given by the commutative diagram
$$
\begin{tikzcd}
\E_{I',I} \arrow[r,"\phi_{I',I}^\sharp"] \arrow[d] & \E_I \arrow[d]\\
\V_{I',I}   \arrow[r,"\phi_{I',I}^\flat"] &  \V_I
\end{tikzcd}
$$
where $\phi^\flat_{I',I}: \V_{I',I}\to \V_I$ is the inclusion and the bundle map $\phi^\sharp_{I',I}$ is defined by canonically identifying the fibers of $\E_{I'}$ and $\E_I$ with
$$e^{r'_1,\ell}\oplus\dots\oplus e^{r'_{k'},\ell} \quad \mbox{ and } \quad e^{r_1,\ell}\oplus\dots\oplus e^{r_{k},\ell},$$
and including
$$e^{r'_1,\ell}\oplus\dots\oplus e^{r'_{k'},\ell}\subset e^{r_1,\ell}\oplus\dots\oplus e^{r_{k},\ell}.$$
Here $r_j= {\bf t}(\mathcal{M}_{i_j})$ and $r_j'={\bf t}(\mathcal{M}_{i'_j})$. We have
\be
\item[(a)] $\phi_{I',I}^\sharp\circ\overline\bdry_{I'}=\overline\bdry_I\circ \phi_{I',I}^\flat$ on $\V_{I',I}$; and
\item[(b)] $\psi_{I}\circ \phi_{I',I}^\flat=\psi_{I'}$ on $\overline\bdry^{-1}_{I'}(0)\cap \V_{I',I}$.
\ee

For $I=(i_1,\dots,i_k)$ we set
\begin{equation} \label{eqn: defn of s I}
{\mathfrak s}_{I}=(\zeta(\mathfrak{nl}_{i_1}) s_{r_1},\dots, \zeta(\mathfrak{nl}_{i_{k-1}})s_{r_{k-1}}, s_{r_k})\in e^{r_1,\ell}\oplus\dots\oplus e^{r_{k},\ell},
\end{equation}
where the function $\zeta$ is as given in Definition~\ref{defn: zeta}.
Denote $\mathfrak{s}_{I',I} := \mathfrak{s}_{I'}|_{\V_{I',I}}$.
It is immediate that
$$\phi_{I',I}^\sharp \circ \mathfrak{s}_{I',I}= \mathfrak{s}_I \circ \phi_{I',I}^\flat$$  and
$$\overline \partial ^{-1}_I (\mathfrak s_I) \cap \phi_{I',I}^\flat(\V_{I',I})= \phi_{I',I}^\flat (\overline \partial ^{-1}_{I'} (\mathfrak s_{I',I})).$$

\subsection{Gluing}
\cbu
The following gluing results can be proven in a manner similar to Theorems 6.4.1, 6.4.2 in \cite{BH2} (in the contact case) and Theorem A.21 in \cite{ES}. \cb

\begin{thm}[Gluing] \label{thm: gluing}
For sufficiently large $R>0$, there exists a gluing map
\begin{equation}
\mathfrak G_{(i_1,\dots,i_m)}:\V_{i_1} \times\dots \times \V_{i_m} \times (R,\infty)^{m-1}\to \V_{(i_1,\dots,i_m)}
\end{equation}
which satisfies the following: Writing $T_1,\dots, T_{m-1}$ for the coordinates on $(R,\infty)^{m-1}$,
\be
\item $\mathfrak G_{(i_1,\dots,i_m)}$ is a $C^1$-diffeomorphism onto its image;
\item $\op{Im} \mathfrak G_{(i_1,\dots,i_m)}\supset \V_{(i_1,\dots,i_m)}\cap \{\mathfrak{nl}_{i_j}\geq R+\varepsilon'', \forall j<m\}$;
\item $\mathfrak G_{(i_1,\dots,i_m)}([u_{i_1}],\dots,[u_{i_m}], T_1,\dots, T_{m-1})$ is \cbu $\sigma$-close to the broken strip $([u_{i_1}],\dots,[u_{i_m}])$ for some $\sigma > 0$ (in the $C^1$-topology; see Definition~\ref{defn: close to breaking}); \cb
\item for $j=1,\dots,m-1$, the functions $(\mathfrak G_{(i_1,\dots,i_m)})_*T_j$ and $\mathfrak{nl}_{i_j}$ are $C^1$-close;
\item $\overline\bdry (\mathfrak G_{(i_1,\dots,i_m)}([u_{i_1}],\dots,[u_{i_m}], T_1,\dots, T_{m-1}))$ and $(\overline\bdry u_{i_1},\dots, \overline\bdry u_{i_m})$, viewed as elements of $e^{r_1,\ell}\oplus\dots\oplus e^{r_{k},\ell}$, $r_j={\bf t}(\mathcal{M}_{i_j})$, are \cbu $C^0$-close; \cb
\item the errors in (3), (4), and (5) go to zero as all $T_j\to \infty$.
\ee
\end{thm}

\begin{thm}[Iterated gluing] \label{thm: iterated gluing}
For sufficiently large $R>0$, there is a gluing map
\begin{align*}
\mathfrak G_{(i_1,\dots, (i_a,\dots,i_b),\dots i_m)}:&  \V_{i_1}\times\dots\times \V_{i_a-1}\times \V_{(i_a,\dots,i_b)}\\
& \times \V_{i_b+1}\times  \dots\times \V_{i_m} \times (R,\infty)^{m-(b-a)-1}\to \V_{(i_1,\dots,i_m)},
\end{align*}
satisfying properties analogous to those of Theorem~\ref{thm: gluing} and
such that
$$\mathfrak G_{(i_1,\dots,i_m)} ~~~\mbox{ and } ~~~\mathfrak G_{(i_1,\dots, (i_a,\dots,i_b),\dots i_m)}\circ (\op{id},\dots, \mathfrak G_{(i_a,\dots,i_b)},\dots, \op{id})$$
are $C^1$-close with error $\to 0$ as all the coordinates of $(R,\infty)^{m-(b-a)-1}$ go to $\infty$.
\end{thm}

\subsection{Equivariant semi-global Kuranishi structures} \label{subsection: equivariant semi global}

The Kuranishi charts constructed in Section~\ref{charts} and \ref{charts'} can be organized into a $G$-invariant semi-global Kuranishi structure.

\cbu
{\em Again, for the moment we work with the $G$-invariant finite set $\{\mathcal{M}_1,\dots, \mathcal{M}_N\}$ of moduli spaces.} \cb

Our definition is similar to McDuff-Wehrheim's treatment of Kuranishi structures (called {\em atlases}) in \cite{MW}. (1)--(3) are general properties of Kuranishi structures/atlases and (4) and (5) are specific ``semi-global" properties.

\begin{defn}[Semi-global Kuranishi structure]
A {\em semi-global Kuranishi structure $\Kur$} is a category consisting of the following data:
\begin{enumerate}
\item The objects are semi-global Kuranishi charts $\mathcal C_I=(\pi_I:\E_I\to \V_I, \overline\partial_I,\psi_I)$:
\be
\item for each $i$, $C_i=(\pi_i: \E_i\to \V_i,\overline\bdry_i,\psi_i)$ is an interior Kuranishi chart for $\K_i\subset \mathcal{M}_i$;
\item for each $I = (i_1,\dots,i_m)$, $\pi_I:\E_I\to \V_I$ is a finite rank vector bundle over a finite-dimensional manifold, $\overline{\partial}_I: \V_I\to \E_I$ is a section, $\psi_I: \overline{\partial}^{-1}_I(0)\to \mathcal M_{c(I)}$ is a homeomorphism onto an open subset of $\mathcal{M}_{c(I)}$, and $\dim \V_I-\op{rk}\E_I=\op{vdim}\mathcal{M}_{c(I)}$; and
\item for each $i$, $\cup_{c(I)=(i)}\op{Im}(\psi_I)=\mathcal{M}_i$.
\ee

\item For each $I'\leq I$ there is a \cbu specified \cb morphism $\phi_{I',I}:\mathcal C_{I'} \to \mathcal C_I$ encoded by the data $(\V_{I',I}, \phi_{I',I}^\sharp, \phi_{I',I}^\flat)$ and given by restriction-inclusion: first restrict $\E_{I'}\to\V_{I'}$ to an open subset $\V_{I',I}\subset \V_{I'}$ and then take the inclusion of vector bundles given by a commutative diagram
$$
\begin{tikzcd}
\E_{I',I}:=\E_{I'}\vert_{\V_{I',I}}   \arrow[r,"\phi_{I',I}^\sharp"] \arrow[d] & \E_I \arrow[d]\\
\V_{I',I}   \arrow[r,"\phi_{I',I}^\flat"] &  \V_I,
\end{tikzcd}
$$
subject to:
\be
\item $\phi_{I',I}^\sharp\circ\overline\bdry_{I'}=\overline\bdry_I\circ \phi_{I',I}^\flat$ on $\V_{I',I}$;
\item $\psi_{I}\circ \phi_{I',I}^\flat=\psi_{I'}$ on $\overline\bdry^{-1}_{I'}(0)\cap \V_{I',I}$;
\item $(\overline\bdry_I)_*:T\V_I\to \E_I$ descends to an isomorphism
$$T\V_I/ (\phi_{I',I}^\flat)_* (T\V_{I',I})\stackrel\sim\longrightarrow \E_I/\E_{I'}.$$
\ee

\item The composition of morphisms is defined so that $ \phi_{I'',I}=\phi_{I'',I'}\circ \phi_{I',I}$.
\ee
The following are {\em strata compatibility conditions}:
\be

\item[(4)] (Neck length functions) For each $(i)<(i_1',i_2')$, there exists a smooth (modified) neck length function
$$\mathfrak{nl}_{(i_1',i_2')}: \cup_{c(I'')=(i)} \widetilde{\mathcal{G}}_\sigma(\V_{i_1''},\dots, \V_{i_k''}) \to\R^{\geq 0}$$
such that
$$\V_{I',I}:=\{[u] \in \V_{I'}~|~\mathcal{L}-\varepsilon''\leq \mathfrak{nl}_{(c(i_1,\dots,i_j),c(i_{j+1},\dots,i_k))}([u])\leq\mathcal{L}, \forall i_j\in \delta(I,I')\}.$$

\item[(5)] For each $I=(I_1,\dots,I_m)$\footnote{Here we abuse notation and refer both $(I_1,\dots, I_m)$ and $(i_{11},\dots, i_{1j_1},\dots, i_{m1},\dots, i_{mj_m})$ by $I$, where $I_k=(i_{k1},\dots,i_{kj_k})$.} there exists a $C^1$-bundle map $(\widetilde{\mathfrak G}_I, \mathfrak G_I)$:
$$
\begin{tikzcd}
\op{pr}_{I_1}^* \E_{I_1}\oplus \cdots \oplus\op{pr}_{I_m}^* \E_{I_m} \arrow[r,"\widetilde{\mathfrak{G}}_{I}"] \arrow[d] & \E_I \arrow[d]\\
\V_{I_1} \times \cdots \times \V_{I_m} \times (R, \infty)^{m-1}   \arrow[r,"\mathfrak{G}_{I}"] &  \V_I,
\end{tikzcd}
$$
where $R\gg 0$, $\op{pr}_{I_k}: \V_{I_1} \times \cdots \times \V_{I_m} \times (R, \infty)^{m-1} \to \V_{I_k}$ is the projection map, $T_j$ is the coordinate for the $j$th $(R, \infty)$ factor, and
\begin{enumerate}
\item $\mathfrak{G}_{I}$ is a $C^1$-diffeomorphism onto its image;
\item $\op{Im} \mathfrak G_{I}\supset \V_{I}\cap \{\mathfrak{nl}_{(c(I_1,\dots,I_j),c(I_{j+1},\dots,I_m))}\geq R+\varepsilon'', \forall j<k\}$;
\item $\mathfrak G_I([u_{I_1}],\dots,[u_{I_m}], T_1,\dots, T_{m-1})$ is close to the broken strip $([u_{I_1}],\dots,[u_{I_m}])$;
\item for $j=1,\dots,m-1$, the functions $(\mathfrak G_{I})_*T_j$ and $\mathfrak{nl}_{(c(I_1,\dots,I_j),c(I_{j+1},\dots,I_m))}$ are $C^1$-close;
\item $\widetilde{\mathfrak G}_I \circ (\overline\partial_{I_1}, \cdots ,\overline \partial_{I_m})$ and $\overline \partial_I \circ \mathfrak G_I$ are $C^1$-close;
\item the errors of (d) and (e) go to zero as $T_j \to \infty$ for all $j=1,\dots,m-1$;
\item  $\mathfrak G_{(I_1,\dots,I_m)}$ and $\mathfrak G_{(I_1,\dots, (I_a,\dots,I_b),\dots I_m)}\circ (\op{id},\dots, \mathfrak G_{(I_a,\dots,I_b)},\dots, \op{id})$ are $C^1$-close with error $\to 0$ as $T_j\to \infty$ for all $j=1,\dots, m-(b-a)-1$.
\end{enumerate}

\end{enumerate}

We say that $\Kur$ is {\em $G$-invariant} if, for each $g\in G$, $g$ induces an isomorphism
$$(\V_I\to \E_I)\stackrel\sim\longrightarrow(\V_{g(I)}\to \E_{g(I)})$$
such that $\overline\bdry_I$, $\psi_I$,  $\mathfrak{nl}_{(i_1',i_2')}$, $\mathfrak{G}_I$ are taken to
$\overline\bdry_{g(I)}$, $\psi_{g(I)}$, $\mathfrak{nl}_{g(i_1',i_2')}$, $\mathfrak{G}_{g(I)}$.

A {\em section of $\Kur$} is a collection $\{{\mathfrak s}_I: \V_I\to \E_I\}_I$ of obstruction sections such that:
\begin{enumerate}
\item $\phi_{I',I}^\sharp \circ \mathfrak{s}_{I',I}= \mathfrak{s}_I \circ \phi_{I',I}^\flat$, where $\mathfrak{s}_{I',I} := \mathfrak{s}_{I}|_{\V_{I',I}}$;
\item $\overline \partial ^{-1}_I (\mathfrak s_I) \cap \phi_{I',I}^\flat(\V_{I',I})= \phi_{I',I}^\flat (\overline \partial ^{-1}_{I'} (\mathfrak s_{I',I}))$;
\item for each $I = (i_1,\cdots, i_m)$, $\widetilde{\mathfrak G}_I\circ (\mathfrak s_{i_1}, \cdots, \mathfrak s_{i_m})$  and ${\mathfrak s_I}\circ \mathfrak G_I$ are $C^1$-close and the error goes to $0$ as $T_j \to \infty$ for all $i = 1, \cdots , m-1$.
\end{enumerate}
\end{defn}

\begin{rmk}
There is no reason to expect the sections $\{{\mathfrak s}_I\}_I$ to be $G$-invariant.  This will be treated in Section~\ref{subsection: curve counting}.
\end{rmk}

\begin{figure}
\begin{tikzpicture}[>=stealth]
\draw[<-](0,0) -- (5,0);
\draw[<-] (0, 0) -- (0, -5);
\draw (0,-2) -- (5,-2)
	(0,-2.5) -- (5, -2.5)
	(2,0) -- (2, -5)
	(2.5, 0) -- (2.5, -5);
\node [above] at (-0.5,0) {$(\infty, \infty)$};
\node [above] at (5,0) {$\mathfrak{nl}_1$};
\node [left] at (0,-5) {$\mathfrak{nl}_2$};
\node[above] at (1.8, 0) {$\mathcal L$};
\node[above] at (2.8, 0) {$\mathcal L- \varepsilon''$};
\node[left] at (0, -1.8) {$\mathcal L$};
\node[left] at (0, -2.5) {$\mathcal L- \varepsilon''$};

\node[above] at (1, -1.4) {$\V_{(1,2,3)}$};
\node[above] at (4, -1.4) {$\V_{(4,3)}$};
\node[above] at (1, -4) {$\V_{(1,5)}$};
\node[above] at (4, -4) {$\V_{6}$};
\end{tikzpicture}
\caption{Corner structure. Suppose that $c(1,2)=4, c(2,3) = 5, c(4,3) = 6 = c(1,5)$.}
\label{fig: corner structure}
\end{figure}
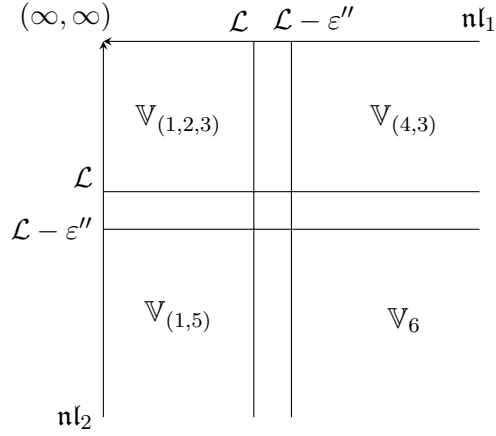

One can also view $\Kur$ as a functor from the index tuple category $\mathcal{I}$ to the ``category of Kuranishi charts".

Let $\Kur (\mathcal{M}_i)$ (also written as $\Kur (p,q;A)$ if $\mathcal{M}_i=\mathcal{M}(p,q;A)$) be the full subcategory of $\Kur$ with objects $I$ such that $c(I)=i$.

Given a section $\mathfrak{S}=\{\mathfrak{s}_I\}_{c(I)=(i)}$ of $\Kur (\mathcal{M}_i)$, we define
$$\mathcal{Z}(\Kur(\mathcal{M}_i),\mathfrak{S})=\left(\coprod_{c(I)=(i)} \overline\bdry^{-1}_I ({\mathfrak s}_I)\right)/\sim_{\Kur},$$
where $\sim_{\Kur}$ is the identification given by the morphisms.

We now come to an important point: {\em There is no reason to expect $\mathcal{Z}(\Kur(\mathcal{M}_i),\mathfrak{S})$ for an abstract semi-global Kuranishi structure to be a manifold, i.e., the Hausdorff property is not automatic.}  However, in our case the existence of the neck length functions implies the following analog of \cite[Lemma 8.8.1]{BH2}:

\begin{lemma}
$\mathcal{Z}(\Kur(\mathcal{M}_i),{\mathfrak S})$ is a manifold.
\end{lemma}

\begin{proof}
Same as that of \cite[Lemma 8.8.1]{BH2}.
\end{proof}

\n \cbu
{\em Taking limits.}  So far we have constructed a semi-global Kuranishi structure $\mathcal{K}$ and a section $\mathfrak{S}$ for the $G$-invariant finite set $\{\mathcal{M}_1,\dots, \mathcal{M}_N\}$. If $\{\mathcal{M}_1,\mathcal{M}_2,\dots\}$ is infinite, we choose increasing sequences $N_1, N_2, \dots \to \infty$ and $\ell_1,\ell_2,\dots \to \infty$ of integers such that $\{\mathcal{M}_1,\dots, \mathcal{M}_{N_j}\}$ is $G$-invariant for each $N_j$ and construct $\mathcal{K}^{(j)}$ such that the fibers of the obstruction bundles $\E_i^{(j)}\to \V_i^{(j)}$ (i.e., $\E_i\to \V_i$ for $j$) are $e^{q, \ell_j}$, where $q={\bf t}(\mathcal{M}_i)$. Since $e^{q,\ell_j}$ naturally includes into $e^{q,\ell_{j+1}}$, there are natural inclusions
$$
\begin{tikzcd}
\E_{I}^{(j)} \arrow[r] \arrow[d] & \E_I^{(j+1)} \arrow[d]\\
\V_{I}^{(j)}   \arrow[r] &  \V_I^{(j+1)}
\end{tikzcd}
$$
that commute with the morphisms $\phi_{I',I}^{(j)}$ and $\phi_{I',I}^{(j+1)}$.  Assuming we have already constructed the section $\mathfrak{S}^{(j)}$, we construct $\mathfrak{S}^{(j+1)}$ such that $\mathfrak{s}_I^{(j+1)}$ is the image of $\mathfrak{s}_I^{(j)}$ under the appropriate inclusions $\oplus_q e^{q,\ell_j} \to \oplus_q e^{q,\ell_{j+1}}$ whenever the entries of $I$ are $\leq N_j$. This is sufficient to ensure that, for $i\leq N_j$, there is a natural identification
$$\mathcal{Z}(\Kur^{(j)}(\mathcal{M}_i),{\mathfrak S}^{(j)})\simeq  \mathcal{Z}(\Kur^{(j+1)}(\mathcal{M}_i),{\mathfrak S}^{(j+1)}).$$
We write $\mathcal{Z}(\Kur(\mathcal{M}_i),{\mathfrak S})$ for any of the $\mathcal{Z}(\Kur^{(j)}(\mathcal{M}_i),{\mathfrak S}^{(j)})$ such that $i\leq N_j$.

\s {\em From now on we will assume $\{\mathcal{M}_1,\mathcal{M}_2,\dots\}$ is finite, making the appropriate modifications as above, if it is not.}

\s\n
{\em Implicit charts.} 
Our semi-global Kuranishi structure $\Kur(\mathcal{M}_i)$, $\mathcal{M}_i=\mathcal{M}(p_i,q_i;A_i)$, can be converted into a {\em single} global implicit chart in the sense of Pardon~\cite{Pa}. Let $\mathcal{S}(p_i,q_i)$ be the set of all the $q_j$ that appear before $(p_i, q_i, A_i)$ in the list \eqref{ordering of triples} and take the global fiber to be 
$$e(p_i,q_i):= \oplus_{r\in \mathcal{S}(p_i,q_i)} e^{r,\ell}.$$ \cb

We consider solutions $(u,\xi)$,
$$u\in \cup_{c(i_1,\dots,i_k)=(i)} \widetilde {\mathcal G}_{\sigma}(\V_{i_1},\dots,\V_{i_k}), \quad \xi=(\xi_r)_{r\in \mathcal{S}(p_i,q_i)}\in e(p_i,q_i),$$
to the equation
\begin{equation}
\overline\bdry u=\sum_{r\in \mathcal{S}(p_i,q_i)} (\zeta\circ \mathfrak{nl}_r(u))\cdot \xi_r,
\end{equation}
where $\zeta$ and $\mathfrak{nl}_r$ are as given in Definitions~\ref{defn: zeta} and \ref{defn: nl two}.  Roughly speaking, we turn off the perturbations for $e^{r,\ell}$ when $\mathfrak{nl}_r(u)\leq \mathcal{L}$ but still remember the data for $e^{r,\ell}$.

\subsection{Equivariant semi-global Kuranishi structures for chain maps and chain homotopies} \label{subsection: chain map and chain homotopy}

\subsubsection{Chain maps} \label{subsubsection: chain maps}

Let $H_s:M\to \R$, $s\in[0,1]$, be a compactly supported, time-dependent, $G$-invariant Hamiltonian function and let $\phi_s$, $s\in[0,1]$, be the corresponding $1$-parameter family of Hamiltonian symplectomorphisms of $(M,\omega)$ with $\phi_0=\op{id}$; we call such a $\phi_s$ a {\em $G$-equivariant Hamiltonian isotopy}.  Writing $L_i'=\phi_1(L_i)$, $i=0,1$, we assume that $L_0'\pitchfork L_1'$.
Let $\{J^s\}_{s\in [0,1]}$ be a $1$-parameter family of almost complex structures that are $\omega$-compatible, $G$-invariant, and satisfy \ref{condition: J}.

Define a smooth function $\vartheta_0:\R\to [0,1]$ such that $\vartheta_0(s)=0$ for $s\leq 0$ and $\vartheta_0(s)=1$ for $s\geq 1$.

Given $p\in L_0\cap L_1$, $q\in L_0'\cap L_1'$, and $A\in \pi_2(p,q)$, let $\mathcal{M}^\circ(p,q;A)$ (we are
suppressing $\{J^s\}$) be the space of smooth maps $u:\R\times[0,1]\to M$ that
satisfy \ref{A3} and \ref{A4}, in addition to:
\be
\item[(A1$'$)] $u_s(s,t) + J^{\vartheta_0(s)}(u(s,t)) u_t(s,t) = 0$, and
\item[(A2$'$)] $u(s,0)\in \phi_{\vartheta_0(s)}(L_0)$ and $u(s,1)\in \phi_{\vartheta_0(s)}(L_1)$.
\ee
\cb
When we are defining chain maps and chain homotopies, the moduli spaces for $(L_0,L_1)$ will have superscripts $-$ as in $\mathcal{M}^-(p,p'; A)$ and the moduli spaces for $(L_0',L_1')$ will have superscripts $+$ as in $\mathcal{M}^+(q,q'; A)$.  

The construction of the Kuranishi charts and the Kuranishi structure from Sections~\ref{charts} to \ref{subsection: equivariant semi global} carry over with very few modifications: Under our assumptions there are finitely many moduli spaces of type $\mathcal{M}^\circ(p,q;A)$, $\mathcal{M}^-(p,q;A)$, and $\mathcal{M}^+(p,q;A)$, which we list as
$$\mathcal{M}_1,\dots, \mathcal{M}_\rho$$
as before so that $\omega(A)$ is in nondecreasing order. The {\em type} of $\mathcal{M}_i$ is given by the superscript $\circ$, $-$, or $+$.

\begin{defn} \label{defn: c index tuple}
A tuple $I=(i_1,\dots,i_k)$ is a {\em $c$-index tuple} (where $c$ stands for chain map), if it satisfies the conditions of Definition~\ref{defn: index tuple} and
\begin{itemize}
\item there exists $i_j$ such that $\mathcal{M}_{i_j}$ has type $\circ$ and all $i_l$ with $l<j$ have type $-$ and all $i_l$ with $l>j$ have type $+$.
\end{itemize}
\end{defn}

The charts $(\pi_I:\E_I\to \V_I,\overline\bdry_I, \psi_I)$ are constructed in exactly the same way as before, where $I$ is now a $c$-index tuple.  By construction the Kuranishi structure is $G$-invariant.

\subsubsection{Chain homotopies} \label{subsubsection: chain homotopies}

Fix $T\gg 0$.  We define a smooth function
$$\Theta:  \R\times[0,1]\to [0,1]$$
with coordinates $(s,\tau)$ for $\R\times [0,1]$ such that:
\begin{itemize}
\item $\Theta(s,0)=1$ for $s\in [-T+1,T-1]$;
\item $\Theta(s,1)=0$ for all $s$;
\item $\Theta(s,\tau)=0$ for all $s>T$ and $s<-T$ and $\tau\in[0,1]$.
\end{itemize}

For each $\tau\in[0,1]$, let $\mathcal{M}^\circ_\tau(p,q;A)$ be the space of smooth maps $u:\R\times[0,1]\to M$ that satisfy \hyperlink{A3}{(A3)}, \hyperlink{A4}{(A4)},
\be
\item[(A1$^\tau$)] $u_s(s,t) + J^{\Theta(s,\tau)}(u(s,t)) u_t(s,t) = 0$, and
\item[(A2$^\tau$)] $u(s,0)\in \phi_{\Theta(s,\tau)}(L_0)$ and $u(s,1)\in \phi_{\Theta(s,\tau)}(L_1)$.
\ee
We also write
$$\mathcal{M}^\circ_{\{\tau\}}(p,q;A)=\coprod_{\tau\in[0,1]}\mathcal{M}^\circ_\tau(p,q;A).$$

For each $c$-index tuple and each $\tau\in[0,1]$ we construct a chart
$$(\pi_{I,\tau}: \E_{I,\tau}\to \V_{I,\tau}, \overline\bdry_{I,\tau}, \psi_{I,\tau}),$$
which can be combined into a family
$$(\pi_{I,[0,1]}: \E_{I,[0,1]}\to \V_{I,[0,1]}, \overline\bdry_{I,[0,1]},\psi_{I,[0,1]}).$$
By construction the family of Kuranishi structures is $G$-invariant.

\section{Orientations}\label{section: orientations}

The goal of this section is to review the definition of a coherent (= compatible with gluing) system of orientations on the moduli space of (finite energy) $J$-holomorphic strips for a pair $(L_0,L_1)$ of Lagrangians, following \cite{FO3} and then adapt it to the case with a $G$-action. \cbu We will see that in general, $g\in G$ only preserves the orientation of $\mathcal M (p,q;A)$ up to a sign $\sigma(g,p,q) \in \{-1, 1\}$ that is independent of $A$.
But this is enough to define a $G$-action on the $CF^\bullet(L_0, L_1).$
\cb

\subsection{Cauchy-Riemann tuples}

A {\em Cauchy-Riemann tuple} \label{section: CR tuple} is a quadruple $(\Sigma, \xi,\eta, D)$ satisfying (CR1)--(CR4):
\be
\item[(CR1)] $\Sigma=B\backslash X$, where $B$ is the closed unit disk in $\C$ and $X$ is a finite subset of $\bdry B$.
\ee
For each $x\in X$, let $I_x\subset \bdry B$ be a small interval neighborhood of $x$ and let $I_{x-}$ and $I_{x+}$ be the two connected components of $I_x \backslash x$.
\be
\item[(CR2)] $\xi$ is a trivial $\mathbb C$-vector bundle over $\overline \Sigma = B$.
\item[(CR3)] $\eta$ is a real subbundle of $\xi|_{\partial \Sigma - X}$ such that $\eta|_{I_{x\pm}}$ extends smoothly to a real subspace $\eta_{x \pm}\subset \xi_x$ over $x$. Moreover, $\xi_x = \eta_{x+} \oplus \eta_{x-}$.
\ee
Let $\Gamma(\Sigma, \xi)$ be the space of compactly supported smooth sections of $\xi|_\Sigma$ that restrict to sections of $\eta$ along $\partial \Sigma \backslash X$.  For each $x\in X$, choose a neighborhood $N(x)\subset B$ and a holomorphic identification of $\Sigma\cap N(x)$ with a strip-like end $[0,\infty)\times[0,1]$ with coordinates $(s,t)$. Let $W^{k+1,p}(\Sigma,\xi)$ be the closure of $\Gamma(\Sigma, \xi)$ in the $W^{k+1,p}$-norm with respect to a metric on $\Sigma$ consistent with the strip-like ends and a metric on $\xi$. The space $W^{k,p}(\Sigma, \wedge^{0,1}\Sigma \otimes_{\mathbb C}\xi)$ is defined similarly.
\be
\item[(CR4)] The operator $D: W^{k+1,p}(\Sigma,\xi) \to W^{k,p}(\Sigma, \wedge^{0,1}\Sigma \otimes_{\mathbb C}\xi)$
is a real-linear Cauchy-Riemann operator such that on each strip-like end
$$Dw = \tfrac{1}{2}(\nabla_s w + J \nabla_t w)\otimes (ds + idt),$$
where $J$ is the complex structure on $\xi$ and $\nabla$ is a connection of $\xi$.
\ee
See \cite[Appendix C]{MS2} for the definition of a real-linear Cauchy-Riemann operator over a compact Riemann surface.

\subsection{Auxiliary orientation data}

Recall the {\em determinant line} of $(\Sigma, \xi,\eta, D)$ is a $1$-dimensional vector space defined by
$$\det D:=\wedge^{\op{top}} \op{ker} D \otimes_\R  \wedge^{\op{top}}(\op{coker}D )^*.$$

Let $\pi:\E(p,q;A) \to \V(p,q;A)$ be an interior semi-global Kuranishi chart for ${\mathcal M}(p,q;A)$. Given $u$ with $[u] \in \V$, we define the Cauchy-Riemann tuple
$$(\Sigma^u, \xi^u, \eta^u, D^u) := (S, u^*TM, \sqcup_{i\in \{0,1\}} u(\cdot,i)^* TL_i, D_u),$$ where $S=\mathbb R \times [0,1]$ and $D^u$ is the linearized $\overline \partial$-operator at $u$.

A coherent system of orientations $\mathfrak{o}(D^u)$ of $\det D^u$ will depend on the following {\em auxiliary orientation data}; see Theorem~\ref{thm: orientation fo3}.

\begin{defn}
A choice of {\em auxiliary orientation data} consists of:
\be
\item[(O1)] \hypertarget{O1} a relative spin structure for the pair $(L_0, L_1)$;
\ee
and for each $p\in L_0 \cap L_1$,
\be
\item[(O2)] \hypertarget{O2} a capping Lagrangian path;
\item[(O3)] \hypertarget{O3} a capping orientation; and
\item[(O4)] \hypertarget{O4} a stable capping trivialization.
\ee
\end{defn}

We will explain (O2) and (O3), leaving (O1) and (O4) for the next subsection.

A {\em capping Lagrangian path} (O2) is a path $\{\mathcal L_{p,t}\}_{0\leq t \leq 1} $ in the oriented Lagrangian Grassmannian $\op{Lag}(T_p M,\omega_p)$ such that $\mathcal L_{p,i} = T_pL_i$ with orientations, for $i = 0, 1$.

For each $p \in L_0 \cap L_1$, we define a Cauchy-Riemann tuple $(\Sigma^{p+}, \xi^{p+}, \eta^{p+}, D^{p+})$ as follows: Let $\Sigma^{p+}$ be the closed unit disk in $\C$ with one boundary puncture, identified with the upper half plane $\H=\{z ~|~ \op{Im} z \geq 0\}$, and let $\pi_p: \Sigma^{p+} \to M$ be the constant map to $p$. We then define:
\begin{itemize}
\item $\xi^{p+} =\pi_p^*(T_pM)$,
\item $\eta^{p+}_z= {\mathcal L}_{p,0}$ for $z\in (-\infty, 0)$, $\eta^{p+}_z={\mathcal L}_{p,z}$ for $z \in [0,1]$, and $\eta^{p+}_z = {\mathcal L}_{p,1}$ for $z \in (1,+\infty)$, and
\item $D^{p+}$ is a fixed real linear Cauchy-Riemann operator (the choice is unique up to homotopy).
\end{itemize}
We can similarly choose the Cauchy-Riemann tuple $(\Sigma^{p-}, \xi^{p-}, \eta^{p-}, D^{p-})$ by swapping the roles of $L_0$ and $L_1$.

Finally, a {\em capping orientation} (O3) is a choice of orientation $\mathfrak{o}(D^{p+})$ (but not $\mathfrak{o}(D^{p-})$).

\subsection{Relative spin structures} \label{subsection: relative spin structures}

A pair $(L_0, L_1)$ is {\em relatively spin} if there exists $st \in H^2(M; \Z/2)$ such that $w_2(TL_i) = \iota_i^* st$ for $i = 0,1$: Fix a triangulation $\tau$ of $M$ such that $L_0$, $L_1$, and $L_0 \cap L_1$ are subcomplexes. Choose an oriented real vector bundle $\nV$ of rank $\geq 2$ on the $3$-skeleton $M^{(3)}$ of $M$ such that $w_2(\nV)= st$. (Here we are using the notation $X^{(i)}$ for the $i$-skeleton of a triangulation of $X$.) Then the bundle $TL_i |_{L_i^{(2)}} \oplus \nV|_{L_i^{(2)}}$ is spin and hence is a trivial bundle. Choosing a spin structure is equivalent to choosing a homotopy class of trivializations ${\mathfrak t}_i$ of $TL_i |_{L_i^{(1)}} \oplus \nV|_{L_i^{(1)}}$ that extends to $L_i^{(2)}$. Since $\pi_2(SO(m))=0$ for $m\geq 3$, the extension to $L_i^{(2)}$ is unique and ${\mathfrak t_i}$ also extends to $L_i^{(3)}$.  We will refer to choices of $\tau$, $V$, and homotopy classes of ${\mathfrak t}_i$, $i=0,1$, as a {\em relative spin structure}; see \cite[Section 8.1]{FO3} for an explanation of when two relative spin structures are equivalent. 
\cbu
A more algebraic (and cleaner) definition of a relative spin structure can be found in \cite[Section 3.1]{WW} and \cite{Sc}.
\cb

Let $\widetilde{\mathcal{L}}_p\to[0,1]$ be a vector bundle whose fiber over $t\in[0,1]$ is $\mathcal{L}_{p,t} \oplus V_p$.
Then a {\em stable capping trivialization} (O4) is a trivialization $\widetilde {\mathfrak t}_p$ of $\widetilde{\mathcal{L}}_p$ that agrees with the trivializations ${\mathfrak t}_i$ of $(TL_i |_{L_i^{(1)}} \oplus \nV|_{L_i^{(1)}})|_p$ that we have already chosen for $i=0,1$.

\subsection{Coherent orientation system}

We review the following theorem from \cite[Section 8.1]{FO3}:

\begin{thm} \label{thm: orientation fo3}
The moduli space of $J$-holomorphic strips admits a coherent orientation system if the pair of Lagrangians $(L_0, L_1)$ is relative spin. Moreover, the choice of auxiliary orientation data \hyperlink{O1}{(O1)}--\hyperlink{O4}{(O4)} determines the orientation.
\end{thm}

We give a sketch of the proof, partly to establish notation.
The fundamental fact that we use is the following (cf.\ \cite[Proposition 34.3]{FO3}), stated without proof.

\begin{fact} \label{prop: orientation for disk}
Given a Cauchy-Riemann tuple $(\Sigma, \mathbb C^n,\eta, D)$,
if $\Sigma$ has no punctures and $\eta$ is trivial, then any trivialization of $\eta$ canonically determines an orientation of $\det D$.
\end{fact}

\begin{proof}[Sketch of proof of Theorem~\ref{thm: orientation fo3}] $\mbox{}$

\s\n {\em Step 1.}
Let $(\Sigma^1,\xi^1,\eta^1,D^1)$ and $(\Sigma^2,\xi^2,\eta^2,D^2)$ be two Cauchy-Riemann tuples.
Given punctures $x_1 \in \partial \Sigma^1$ and $x_2 \in \partial \Sigma^2$, suppose there is a $\C$-linear isomorphism $\Phi: \xi^1_{x_1} \stackrel\sim\longrightarrow \xi^2_{x_2}$
that maps $\eta^1_{x_1\pm}$ to $\eta^2_{x_2 \mp}$.
Then there is an associated Cauchy-Riemann tuple $(\Sigma^{1,2},\xi^{1,2},\eta^{1,2},D^{1,2})$ defined by a straightforward pregluing which identifies $x_1$ and $x_2$ and the orientations of $\det D^1$ and $\det D^2$ induce an orientation of $\det D^{1,2}$.

In particular, if we preglue $(\Sigma^{q+}, \xi^{q+}, \eta^{q+}, D^{q+})$ and $(\Sigma^{q-}, \xi^{q-}, \eta^{q-}, D^{q-})$, we obtain the Cauchy-Riemann tuple $(\Sigma^{q+,q-}, \xi^{q+,q-}, \eta^{q+,q-}, D^{q+,q-})$ and it has a canonical orientation by Fact~\ref{prop: orientation for disk}. (Here we are taking the trivializations of $\eta^{q+}$ and $\eta^{q-}$ to come from the same trivialization of $\mathcal{L}_p$; then the trivialization of $\eta^{q+,q-}$ is independent of the choice of trivialization of $\mathcal{L}_p$.)  Hence the capping orientation $\mathfrak{o}(D^{q+})$ determines $\mathfrak{o}(D^{q-})$.

For any $u$ with $[u] \in \V(p,q, A)$, we preglue
$$(\Sigma^{p+}, \xi^{p+}, \eta^{p+}, D^{p+}),\quad (\Sigma^u, \xi^u, \eta^u, D^u),\quad (\Sigma^{q-}, \xi^{q-}, \eta^{q-}, D^{q-})$$
along $p$ and $q$ to obtain
$$(\Sigma^{p+,u,q-}, \xi^{p+,u,q-}, \eta^{p+,u,q-}, D^{p+,u,q-}).$$
If we can orient $\det (D^{p+,u,q-})$, then $\mathfrak{o}(D^u)$ is determined by ${\mathfrak o}(D^{p+,u,q-})$ and the capping orientations $\mathfrak{o}(D^{p+})$ and $\mathfrak{o}(D^{q-})$.

\s\n {\em Step 2.}
By the simplicial approximation theorem, after a homotopy we can assume that $u(\Sigma) \subseteq M^{(2)}$ and $u (\partial \Sigma) \subseteq L_0^{(1)}\cup L_1^{(1)}$.  Let $V\to M^{(3)}$ be the bundle appearing in the definitions of (O1) and (O4).

Define Cauchy-Riemann tuples
$$(\Sigma^{p+}, \xi_{\nV}^{p+}, \eta_{\nV}^{p+}, D_{\nV}^{p+}), \quad(\Sigma^u, \xi_{\nV}^u, \eta_{\nV}^u, D_{\nV}^u), \quad (\Sigma^{p-}, \xi_{\nV}^{p-}, \eta_{\nV}^{p-}, D_{\nV}^{p-})$$
in the same way as the versions without $V$, except that we replace $TM$ by $\nV \oplus i \nV$, $TL_i$ by $V$ for $i=0,1$, and $\mathcal L_{p,t}$ by $\nV_p$.
By pregluing as in Step 1, we obtain
$$(\Sigma^{p+,u,q-}, \xi_{\nV}^{p+,u,q-}, \eta_{\nV}^{p+,u,q-}, D_{\nV}^{p+,u,q-}).$$
A key point to observe now is that, since $\nV$ is oriented and defined over $u(\Sigma)$, there is a canonical equivalence class of trivializations of $\eta_{\nV}^{p+,u,q-}$ and hence a canonical orientation of $\det D_{\nV}^{p+,u,q-}$ by Fact~\ref{prop: orientation for disk}.

We take the direct sum of
$$(\xi^{p+,u,q-}, \eta^{p+,u,q-}, D^{p+,u,q-}) \quad \mbox{ and } \quad (\xi_{\nV}^{p+,u,q-}, \eta_{\nV}^{p+,u,q-}, D_{\nV}^{p+,u,q-}),$$
over $\Sigma^{p+,u,q-}$ to obtain
$$(\Sigma^{p+,u,q-}, \xi^{p+,u,q-}\oplus \xi_{\nV}^{p+,u,q-}, \eta^{p+,u,q-}\oplus \eta_{\nV}^{p+,u,q-}, D^{p+,u,q-}\oplus D_{\nV}^{p+,u,q-}).$$
Now $\mathfrak{t}_i$, $\widetilde{\mathfrak t}_p$, and $\widetilde{\mathfrak t}_q$ give a trivialization of $\eta^{p+,u,q-}\oplus \eta_{\nV}^{p,u,q-}$, so  $\det (D^{p+,u,q-}\oplus D_{\nV}^{p+,u,q-})$ is canonically oriented by Fact~\ref{prop: orientation for disk}. Since $\det (D^{p+,u,q-}\oplus D_{\nV}^{p+,u,q-})$ is canonically isomorphic to $\det D^{p+,u,q-}\otimes \det D_{\nV}^{p+,u,q-}$ and $\det D_{\nV}^{p+,u,q-}$ is canonically orientated, we obtain a canonical orientation of $\det D^{p+,u,q-}$.

\s\n {\em Step 3.}
It remains to show that $\mathfrak{o}(D^u)$ is independent of the choices.  We refer the reader to \cite[Section 8.1]{FO3} for a proof.
\end{proof}

Since $\det D^u$ is canonically isomorphic to $\det \mathbb D^u$, where $\mathbb D^u$ is the linearized operator of $\overline\partial_J: \V(p,q;A) \to \E(p,q;A)$, a choice of auxiliary orientation data induces a system of orientations on
$$(\Lambda^{\op{top}} \E_I)^*\otimes \Lambda^{\op{top}}T \V_I.$$

Next we study orientations under the group action. To do that, we first need to allow $G$ to act on the obstruction bundle.

\subsection{Orientations on $e^{r,\ell}$} \label{subsection: orientation of fiber}

\begin{lemma} \label{lemma: orientations on e}
If $\ell$ is an even multiple of $n$, then $e^{r,\ell}$ admits a canonical $G$-invariant orientation.
\end{lemma}

\begin{proof}
Without loss of generality, we assume:
\begin{itemize}
\item[(i)] $T_rM\simeq \R^n \oplus i\R^n=\C^n$, where $T_r L_0$ is the $\R^n$ factor and $T_r L_1$ is the $i\R^n$ factor;
\item[(ii)] $J(r)=J_0$ is the standard complex structure that takes $v\in \R^n$ to $iv\in i\R^n$ and $g_r$ is the standard Euclidean structure on $T_rM$; and
\item[(iii)] $G$ leaves $T_r L_0$ invariant.
Since $G$ is compatible with $J$ and $g$, it can be described by a representation $\rho: G\to O(\R^n)$.
\end{itemize}
The asymptotic operator $A$ is given by $-J_0 {\frac{\bdry}{\bdry t}}$ with boundary conditions $\R^n$ at $t=0$ and $i\R^n$ at $t=1$. For each $k=0,1,\dots$, there are $n$ eigenfunctions
$$\tilde e^k_j: [0,1]\to \C^n, \quad t\mapsto e^{i(\pi k + \pi/2)t }e_j,\quad j=1,\dots, n,$$
where $e_1,\dots, e_n$ is a basis for $\R^n$.  Writing $\ell= 2k_0n$, we choose the orientation
\begin{equation}\label{eqn: or}
\tilde e^0_1\wedge \dots \wedge\tilde e^0_n \wedge \dots \wedge \tilde e^{2k_0-1}_1\wedge\dots\wedge \tilde e^{2k_0-1}_n
\end{equation}
for $e^{r,\ell}$. Since $G$ acts on each $\R\langle \tilde e^k_1,\dots,\tilde e^k_n\rangle$ in the same way as on $\R^n$ using the identification $\tilde e^k_j\mapsto e_j$, for any $g\in G$,
$$g(\tilde e^0_1)\wedge \dots \wedge g(\tilde e^0_n) \wedge \dots \wedge g(\tilde e^{2k_0-1}_1)\wedge\dots\wedge g(\tilde e^{2k_0-1}_n)=\tilde e^0_1\wedge \dots \wedge\tilde e^0_n \wedge \dots \wedge \tilde e^{2k_0-1}_1\wedge\dots\wedge \tilde e^{2k_0-1}_n$$
and $G$ preserves the orientation.  Note that the definition in Equation~\eqref{eqn: or} does not depend on the orientation of $\R^n$.
\end{proof}

{\em From now on let us assume that $\ell$ is an even multiple of $n$ and hence all the $e^{r,\ell}$ are canonically oriented,} so $G$ acts on $\mathbb E_I \to \mathbb V_I$.

\subsection{Orientations under group action} \label{relative spin under G}

Now we study the action of $G$ on the orientation of $(\Lambda^{\op{top}} \E_I)^*\otimes \Lambda^{\op{top}}T \V_I$.

\cbu
{\em We assume Condition \ref{condition: O} from Section~\ref{section: introduction}, i.e., that the relative spin structure is preserved under $G$}, whose definition we give presently.

Let $(\tau, V, \mathfrak t_0, \mathfrak t_1)$ be a relative spin structure for $(L_0, L_1)$. Let $\tau$ be a $G$-equivariant triangulation of $M$; such a triangulation exists by the equivariant triangulation theorem. Then $(\tau, V, \mathfrak t_0, \mathfrak t_1)$ is {\em preserved by $G$}, if for any $g\in G$, there exists an orientation-preserving bundle isomorphism $\theta_g: V \stackrel\sim\longrightarrow V$ such that
\begin{itemize}
\item $\pi_V\circ \theta_g = g \circ \pi_V$, where $\pi_V: V\to M^{(3)}$ is the projection to the base, and
\item for each $i=0,1$, the trivialization $$\mathfrak t_i:TL_i |_{L_i^{(2)}} \oplus \nV|_{L_i^{(2)}} \to L_i^{(2)} \times \R^{d}$$ is homotopic to
$$g_\sharp \mathfrak t_i := \mathfrak t_i \circ (g_* \oplus \theta_g |_{L_i^{(2)}})^{-1},$$ 
where $d = n + \op{rank} V$  and 
$$g_* \oplus \theta_g|_{L_i^{(2)}}: TL_i |_{L_i^{(2)}} \oplus \nV|_{L_i^{(2)}} \to TL_i |_{L_i^{(2)}} \oplus \nV|_{L_i^{(2)}}.$$
\end{itemize}
\cb

%

For $p\in L_0 \cap L_1$ and $g\in G$, let $s = g p$. At $s$, we have the canonical isomorphism
\begin{equation} \label{eqn: direct sum at s}
\det D^{s+} \otimes \det D^{s-}\otimes \det D_{\nV}^{s+,s-} \simeq  \det (D^{s+, s-} \oplus D_{\nV}^{s+,s-})
\end{equation}
coming from gluing. Let $\mathfrak o(D^{p+})$ and $\mathfrak o(D^{s-})$ be the capping orientations of $\det D^{p+}$ and $\det D^{s-}$ and let $\mathfrak o( D^{s+,s-}_{\nV})$ be the canonical orientation of $D^{s+,s-}_{\nV}$. Then $g_{\#} \mathfrak o(D^{p+}) \otimes \mathfrak o(D^{s-}) \otimes \mathfrak o( D^{s+,s-}_{\nV})$ determines an orientation of the left-hand side of Equation ~\eqref{eqn: direct sum at s}.
On the right-hand side of Equation~\eqref{eqn: direct sum at s}, we have a canonical orientation of $\det (D^{s+, s-} \oplus D_{\nV}^{s+,s-})$ coming from the concatenation of the stable capping trivializations $g_{\#}\widetilde{\mathfrak t}_p$ and $\widetilde{\mathfrak t}_s$. (The trivializations $g_{\#}\widetilde{\mathfrak t}_p$ and $\widetilde{\mathfrak t}_s$ a priori do not agree at the endpoints.  We assume that $g_{\#}{\mathfrak t}_i$ has been homotoped to ${\mathfrak t}_i$ and by abuse of notation we refer to $g_{\#}\widetilde{\mathfrak t}_p$ as the result of applying the homotopy to  $g_{\#}\widetilde{\mathfrak t}_p$.)  We compare these two orientations via the isomorphism of Equation~\eqref{eqn: direct sum at s}, and define $\sigma(p,g) \in \{\pm 1\}$ to be the difference.  For $u$ with $[u]\in \V(p,q,A)$, let $\mathfrak o(D^u)$ be the orientation of $u$ determined by the auxiliary orientation data \hyperlink{O1}{(O1)}--\hyperlink{O4}{(O4)}.  Then one can check that  $g_{\#} \mathfrak o(D^u) = \sigma(p,g) \sigma(q,g) \mathfrak o (D^{gu})$.

In general, $g\in G$ may not preserve the orientation, but we can define the action of $g\in G$ on $CF^{\bullet}(L_0, L_1)$ by sending $[p]$ to $\sigma(p,g)[gp]$.
In the case when the moduli spaces that we count to define the differential $d$ of $CF^{\bullet}(L_0, L_1)$ are $G$-invariant, it is obvious that the $G$-action on $CF^{\bullet}(L_0, L_1)$ commutes with $d$.
In Section \ref{section: equivariance}, we see this is true even when the moduli space is not $G$-invariant.

{\em From now on, we fix a choice of auxiliary orientation data \hyperlink{O1}{(O1)}-\hyperlink{O4}{(O4)} such that the relative spin structure (O1) is preserved under the $G$-action.} This gives an orientation of $(\Lambda^{\op{top}} \E_I)^*\otimes \Lambda^{\op{top}}T \V_I$.
Since the fiber of $\E_I\to \V_I$ is canonically oriented by Section~\ref{subsection: orientation of fiber}, we also get an orientation of $\V_I$.

\section{Equivariance of curve counting} \label{section: equivariance}

\subsection{Equivariance of curve counting} \label{subsection: curve counting} $\mbox{}$


\s\n
{\em Choice of $\mathfrak{S}$.}
We first describe how to choose $\mathfrak{S}=\{\mathfrak{s}_I\}_I$ to be {\em as $G$-equivariant as possible}.  First decompose $L_0\cap L_1$ into a disjoint union of $G$-orbits $\mathcal{O}_p$, $p\in L_0\cap L_1$.  Given $\mathcal{O}_p$, pick a generic $s_p\in e^{p,\ell}$ which is sufficiently close to the origin and for each $q\in \mathcal{O}_p$ choose a single $g\in G$ such that $g(p)=q$ and set $s_q= g(s_p)\in e^{q,\ell}$.  We then choose $\mathfrak{s}_i=s_p\in e^{p,\ell}$, where $p=\mathbf{t} (\mathcal{M}_i)$, and construct $\mathfrak{s}_I$ as described in Sections~\ref{charts} and \ref{charts'}. We additionally assume that:
\be  \hypertarget{star}{}
\item[(*)] $|s_i| \ll |s_j|$ if $i > j$.
\ee

\begin{rmk}
Note that $\mathfrak{S}$ is not expected to equal $g(\mathfrak{S})$ for all $g\in G$.  If we replace $\mathfrak{S}$ by a $G$-equivariant collection of {\em multisections}, the Floer chain groups will be defined over $\Q$ as in Cho-Hong~\cite{CH}. Since this leads to some loss of information, we choose to work with collections of sections.
\end{rmk}

The following key theorem makes the equivariant count work.

\begin{thm} \label{thm: equivariant count}
If $\op{vdim}\mathcal{M}_i=0$ and $g(\mathcal{M}_i)=\mathcal{M}_i$, then $\mathcal{Z}(\Kur(\mathcal{M}_i),{\mathfrak S})$ and $\mathcal{Z}(\Kur(\mathcal{M}_i),g({\mathfrak S}))$ are cobordant.
\end{thm}

\begin{proof}
If $\overline{\mathcal{M}}_i=\mathcal{M}_i$, i.e., there is no boundary, then $\mathcal{Z}(\Kur(\mathcal{M}_i),{\mathfrak S})$ is given by the preimage of $s_{r}$, $r=\mathbf{t}(\mathcal{M}_i)$, under the map
$$\overline\bdry_i:\V_i\to e^{r,\ell}.$$
Similarly, $\mathcal{Z}(\Kur(\mathcal{M}_i),g({\mathfrak S}))$ is given by the preimage of $g(s_{r})$. Since $\overline\bdry_i (\bdry \V_i)$ does not contain $0$, and $\mathfrak S$ and $g(\mathfrak S)$ are sufficiently close to $0$, the two preimages are cobordant.

The main point of the proof is to homotop ${\mathfrak S}$ to $g({\mathfrak S})$ near $\bdry \mathcal{M}_i$ (i.e., for curves in $\V_i$ that are close to breaking) when it is nonempty. In order to simplify the cumbersome notation, let us assume without loss of generality that:
\be
\item[(**)] $\mathcal{M}_{i_j}=g(\mathcal{M}_{i_j})$ for all $i_j$ that appears in $I=(i_1,\dots,i_m)$, $m\geq 2$, such that $c(I)=(i)$ and for all $g\in G$; in particular $\mathcal{M}_i$ is $G$-invariant.
\ee

\s\n
{\em Step 1.}
Given $I=(i_1,\dots,i_m)$  such that $c(I)=(i)$, consider the composition
$$\V_{i_1}\times\dots \times \V_{i_m}\times (R,\infty)^{m-1}\xrightarrow{\mathfrak{S}_{(i_1,\dots,i_m)}} \V_{I}\stackrel{\overline\bdry_I}\longrightarrow e^{r_1,\ell}\oplus \dots\oplus e^{r_m,\ell},$$
where $r_j=\mathbf{t}(\mathcal{M}_{i_j})$ and in particular $r_m = r$.
As $R\to\infty$, its image approaches the image of the product map
$$(\overline\bdry_{i_1},\dots,\overline\bdry_{i_m}):\V_{i_1}\times \dots \times \V_{i_m}\to e^{r_1,\ell}\oplus \dots\oplus e^{r_m,\ell}.$$
This implies that $\op{Im}(\overline\bdry_I\circ \mathfrak{S}_{(i_1,\dots,i_m)})$ is effectively $\op{Im}(\overline\bdry_{i_1},\dots,\overline\bdry_{i_m})$.
We assume that the generic point $(s_{r_1},\dots,s_{r_m})\in e^{r_1,\ell}\oplus \dots\oplus e^{r_m,\ell}$ has been chosen to avoid $\op{Im}(\overline\bdry_{i_1},\dots,\overline\bdry_{i_m})$.  Note that under our assumption $\op{vdim}\mathcal{M}_i=0$, we have
\begin{equation}\label{eqn: sum of dimensions}
(m-1)+\sum_{j=1}^m \dim \V_{i_j}=\dim \V_I=\sum_{j=1}^m \dim e^{r_j,\ell}.
\end{equation}
\begin{rmk}
We will see that $\mathcal{Z}(\Kur(\mathcal{M}_i),{\mathfrak S})$ and $\mathcal{Z}(\Kur(\mathcal{M}_i),g({\mathfrak S}))$ are empty sets ``near the boundary" unless $m=2$ and  $(\op{vdim} \mathcal{M}_{i_1},\op{vdim}\mathcal{M}_{i_2})=(-1,0)$ or $(0,-1)$.
\end{rmk}
\cb
We now continue the proof in steps based on the value of $m$.

\s\n
{\em Step 2.}  Suppose that $m=2$.

\s\n
{\em Step 2A.} Suppose that $(\op{vdim} \mathcal{M}_{i_1},\op{vdim}\mathcal{M}_{i_2})=(0,-1)$ or $(-1,0)$.  We treat the former; the latter is analogous. Consider the $G$-equi\-variant, codimension one map
$$\overline\bdry_{i_2}: \V_{i_2}\to e^{r_2,\ell}.$$
Let $S_{\rho_2}^{r_2,\ell-1}\subset e^{r_2,\ell}$ (resp.\ $B_{\rho_2}^{r_2,\ell}\subset e^{r_2,\ell}$) be a sphere (resp.\ an open ball) of radius $0<\rho_2 \ll \op{dist}(\{0\},\overline\bdry_{i_2}(\op{\bdry}\V_{i_2}))$.  The action $G\to \op{GL}(e^{r_2,\ell})$ factors through the orthogonal group and hence $G$ acts on $S_{\rho_2}^{r_2,\ell-1}$.

\begin{lemma} \label{lemma: zero signed intersection}
If $s_{r_2}\in e^{r_2,\ell}$ is a point such that $0< |s_{r_2}|<\rho_2$ and $s_{r_2}\not\in \op{Im}(\overline\bdry_{i_2})$, then for any path $\gamma_{r_2}:[0,1]\to B_{\rho_2}^{r_2,\ell}$ from $s_{r_2}$ to $g(s_{r_2})$, the signed intersection number $\langle \gamma_{r_2},\overline\bdry_{i_2}\rangle$ between $\gamma_{r_2}$ and $\overline\bdry_{i_2}$ is zero.
\end{lemma}

\begin{proof}[Proof of Lemma~\ref{lemma: zero signed intersection}]
We may slightly perturb $\rho_2$ such that $S_{\rho_2}^{r_2,\ell-1}\pitchfork \overline\bdry_{i_2}$.  Then $N:= \overline\bdry_{i_2}^{-1}(S_{\rho_2}^{r_2,\ell-1})$ is a submanifold of $\V_{i_2}$ of dimension $(\ell - 2)$.  We homotop $\gamma_{r_2}$ to a concatenation $\gamma_1\gamma_2\gamma_3$, where
\be
\item $\gamma_1$ is a slightly perturbed radial ray from $s_{r_2}$ to a point $x_1\in C:=S_{\rho_2}^{r_2,\ell-1}-\overline\bdry_{i_2} N$;
\item $\gamma_2$ connects $x_1$ to $g(x_1)$ on $S_{\rho_2}^{r_2,\ell-1}$; and
\item $\gamma_3=(g(\gamma_1))^{-1}$ from $g(x_1)$ to $g(s_{r_2})$.
\ee
See Figure~\ref{fig: concat}. The contributions to $\gamma_{r_2}\cap \overline\bdry_{i_2}$ from $\gamma_1$ and $\gamma_3$ cancel, and it remains to calculate the contribution from $\gamma_2$.

There exists a locally constant weight function $w: C\to \Z,$
such that the values on adjacent connected components differ by $1$;
more precisely, given any two points $x, x' \in C$, if $\delta$ is a path from $x$ to $x'$ in $S_{\rho_2}^{r_2,\ell-1}$ and $\delta$ intersects $\overline\bdry_{i_2}|_N$ positively and only once, then $w(x)-w(x')=1$.  The existence of such a function follows from the existence of the winding number of the map
$$\overline\bdry_{i_2}|_N: N\to S_{\rho_2}^{r_2,\ell-1}-\{z\}\simeq \R^{\ell-1},$$
for any $z\in C$.
More precisely, for any $x \in \R^{\ell - 1} \backslash \overline\bdry_{i_2}(N)$,
$w(x)$ is given by the degree of the mapping from $N$ to $ \R^{\ell - 1}\backslash \{x\} \cong S^{\ell-2}.$
Any two weight functions differ by an integer-valued constant function (depending on the choice of $z$).

Next we claim that $w=w\circ g$ for any $g\in G$. First observe that $w\circ g$ is also a weight function. Arguing by contradiction, suppose there is a component $C_0$ of $C$ such that $w(g(C_0))=w(C_0)+k$, $k\not=0$.  Then $w\circ g =w+k$, and $w(g^2(C_0))=(w \circ g)(C_0)+k = w(C_0) + 2k$.  Applying this procedure to the order $m$ of the group $G$, $w(C_0)=w(g^m(C_0))=w(C_0)+mk$, which is a contradiction.

Since
$\langle \gamma_2, \overline\bdry_{i_2}\rangle= \langle \gamma_2,\overline\bdry_{i_2}|_N\rangle^\circ,$
where $\langle \cdot,\cdot\rangle^\circ$ is the intersection number on $S_{\rho_2}^{r_2,\ell-1}$, and
$$\langle \gamma_2,\overline\bdry_{i_2}|_N\rangle^\circ= w(x_1)-w(g(x_1))=0,$$
the lemma follows.
\end{proof}

\begin{figure}
\begin{tikzpicture}[>=stealth]
\filldraw (3,3) circle [radius = 0.02];
\node [right] at (3,3) {$O$};
\filldraw[red] (2.5, 3.3) circle [radius = 0.05];
\node[above, red] at (2.5, 3.3) {$s_{r_2}$};

\filldraw[red] (3, 2.4) circle [radius = 0.05];
\node[right, red] at (3, 2.4) {$gs_{r_2}$};

\draw[->, brown]  (3,5) arc [radius=2, start angle = 90, end angle = 450];

\draw[->, thick] (2, 5) .. controls (2.5, 4.5) and (2,3) .. (0.8, 2);
\begin{scope}[xshift = 200 pt, yshift = 55 pt]
\draw [->, thick, rotate = 120] (2, 5) .. controls (2.5, 4.5) and (2,3) .. (0.8, 2);
\end{scope}

\begin{scope}[xshift = 55 pt, yshift = 200 pt]
\draw [->, thick, rotate = 240] (2, 5) .. controls (2.5, 4.5) and (2,3) .. (0.8, 2);
\end{scope}

\draw [->, blue, thick] (2.5, 3.3) -- (1.2, 3.9);
\draw [->, blue, thick] (1.2, 3.9) arc [radius = 2, start angle = 153, end angle = 277];
\draw [->, blue, thick] (3.2,1) -- (3, 2.4);
\node[right, brown] at (4.5,4.5) {$S_{\rho_2}^{r_2, \ell -1}$};
\node[right] at (5, 3) {$\op{Im}\overline \partial _{i_2}$};
\node[above, blue] at (1.5, 3.2) {$\gamma_1$};
\node[above, blue] at (1.22, 1.3) {$\gamma_2$};
\node[right, blue] at (3.07, 1.5) {$\gamma_3$};
\end{tikzpicture}
\caption{}
\label{fig: concat}
\end{figure}

We now explain how to homotop the section $\mathfrak S$ ``near the boundary of" $\V_i$ to another section $\mathfrak S'$ such that:
\be
\item[(M1)] $\mathfrak{S}'$ and $g(\mathfrak{S})$ agree ``near the boundary''; and
\item[(M2)]  $\mathfrak{S}$ and $\mathfrak{S}'$ have the same signed count of intersections with $\overline \partial$.
\ee

In the $m=2$ case, this means that we homotop $\mathfrak{s}_{(i_1,i_2)}$ to another section $\mathfrak{s}'_{(i_1,i_2)}$ such that:
\be
\item $\mathfrak{s}'_{(i_1,i_2)}$ and $g(\mathfrak{s}_{(i_1,i_2)})$ agree ``near the boundary''; and
\item  $\mathfrak{s}_{(i_1,i_2)}$ and $\mathfrak{s}'_{(i_1,i_2)}$ have the same signed count of intersections with $\overline \partial _{(i_1,i_2)}$.
\ee
Pick $\mathcal{L}'\gg \mathcal{L}$ and let $\tau:[\mathcal{L}',\infty)\to [0,1]$ be a smooth function such that
\begin{itemize}
\item $\tau([\mathcal{L}',\mathcal{L}'+\varepsilon''])=0$,
\item $\tau([\mathcal{L}'+2\varepsilon'',\infty))=1$, and
\item its restriction to $(\mathcal{L}'+\varepsilon'', \mathcal{L}'+2\varepsilon'')$ is a diffeomorphism onto $(0,1)$.
\end{itemize}
Let $\gamma^*_{r_j}=\gamma_{r_j}\circ \tau$,
where we take $\gamma_{r_1}$ to be an arbitrary path in $e^{r_1, \ell}$ connecting $s_{r_1}$ to $g(s_{r_1})$ and $\gamma_{r_2}$ to be as in Lemma~\ref{lemma: zero signed intersection}.  We then define
\begin{gather*}
\mathfrak{s}'_{(i_1,i_2)}=\left\{
\begin{array}{cl}
(\gamma^*_{r_1}(\mathfrak{nl}_{i_1}),\gamma^*_{r_2}(\mathfrak{nl}_{i_1})), & \mbox{ for } \mathfrak{nl}_{i_1}\geq \mathcal{L}',\\
\mathfrak{s}_{(i_1,i_2)}, & \mbox{ for } \mathfrak{nl}_{i_1}\leq \mathcal{L}'.
\end{array}
\right.
\end{gather*}

\begin{figure}
\begin{tikzpicture}
\draw[->] (-2.5,0) -- (7,0);
\draw[thick, brown](-2.5, 0) -- (4,0);
\draw[thick, brown](6,1) -- (7,1);
\draw[thick, brown] (4,0) .. controls (4.5,0) and (5.5,1) .. (6,1);
\node[below] at (-2,0) {$\mathcal L$};
\node[below] at (2,0) {$\mathcal L'$};
\node[below] at (4,0) {$\mathcal L' + \varepsilon$};
\node[left] at (5,0.5) {$\tau$};
\node[below] at (6,0) {$\mathcal L' + 2\varepsilon$};
\filldraw (-2,0) circle [radius = 0.02];
\filldraw (2,0) circle [radius = 0.02];
\filldraw (4,0) circle [radius = 0.02];
\filldraw (6,0) circle [radius = 0.02];

\end{tikzpicture}
\caption{}
\end{figure}
By Lemma~\ref{lemma: zero signed intersection}, $\mathfrak{S}$ and $\mathfrak{S}'$ have the same signed count of intersections with $\overline \partial$ near $\V_{(i_1,i_2)}$.

\s\n
{\em Step 2B.} Suppose that $(\op{vdim} \mathcal{M}_{i_1},\op{vdim}\mathcal{M}_{i_2})=(a, -a-1)$ or $(-a-1,a)$ with $a>0$; we treat the former. By Equation~\eqref{eqn: sum of dimensions}, a generic path $\gamma_{r_2}$ from $s_{r_2}$ to $g(s_{r_2})$ does not intersect $\overline\bdry_{i_2}$ and the same construction of $\mathfrak{S}'$ applies.  This covers the homotopy of $\mathfrak{S}$ near $\V_I$ for $m=2$.

\s\n
{\em Step 3.} Suppose $m=3$.  Let $r_j={\bf t}(\mathcal{M}_{i_j})$, $j=1,2,3$.

\s\n
{\em Step 3A.}
Suppose that $\op{vdim} \mathcal{M}_{c(i_2,i_3)}\leq -1$.

\s
Now we have the following variant of Lemma~\ref{lemma: zero signed intersection}:

\begin{lemma} \label{lemma: zero signed intersection 2}
There exists $\rho_3>0$ small such that if $s_{r_3}\in e^{r_3,\ell}$ is a point such that $0< |s_{r_3}|<\rho_3$ and $s_{r_3}\not\in \op{Im}(\overline\bdry_{c(i_2,i_3)})$, then there exists a path $\gamma_{r_3}:[0,1]\to B_{\rho_3}^{r_3,\ell}$ from $s_{r_3}$ to $g(s_{r_3})$ such that:
\be
\item the signed intersection number $\langle \gamma_{r_3},\overline\bdry_{c(i_2,i_3)}\rangle$ is zero and
\item $\gamma_{r_3}$ is disjoint from $\overline\bdry_{c(i_2,i_3)} (\bdry \V_{c(i_2,i_3)})$.
\ee
\end{lemma}

\begin{proof}[Proof of Lemma~\ref{lemma: zero signed intersection 2}] $\mbox{}$

\s\n
{\em Case $\op{vdim} \mathcal{M}_{c(i_2,i_3)}=-1$.} In this case the proof follows the same outline as that of Lemma~\ref{lemma: zero signed intersection}, but $$N:=\overline\bdry_{c(i_2,i_3)}^{-1}(S^{r_3,\ell-1}_{\rho_3})$$
is now a manifold with boundary.  Let us write $N=N'\cup N''$, where $N'$ is closed and each component of $N''$ has nonempty boundary. Writing $\gamma_{r_3}$ as $\gamma_1\gamma_2\gamma_3$ as before,
$$\langle\gamma_2,\overline\bdry_{c(i_2,i_3)}\rangle=\langle \gamma_2,\overline\bdry_{c(i_2,i_3)}|_N\rangle^\circ=\langle\gamma_2,\overline\bdry_{c(i_2,i_3)}|_{N'} +\overline\bdry_{c(i_2,i_3)}|_{N''}\rangle^\circ,$$
where $\langle \cdot,\cdot\rangle^\circ$ is the intersection number on $S^{r_3,\ell-1}_{\rho_3}$. As before, $\langle\gamma_2,\overline\bdry_{c(i_2,i_3)}|_{N'}\rangle^\circ=0$.  We can modify $\gamma_2$ if $\langle\gamma_2,\overline\bdry_{c(i_2,i_3)}|_{N''}\rangle^\circ=k$ by concatenating it with a loop in $S^{r_3,\ell-1}_{\rho_3}$ that winds $-k$ times around $\overline\bdry_{c(i_2,i_3)}|_{\bdry N''}$.  The resulting $\gamma_2$ will have zero signed intersection with $\overline\bdry_{c(i_2,i_3)}|_{N''}$, implying (1).  (2) is immediate since $\overline\bdry_{c(i_2,i_3)}|_{\bdry \V_{c(i_2,i_3)}}$ is a codimension two map.

\s\n
{\em Case $\op{vdim} \mathcal{M}_{c(i_2,i_3)}<-1$.} In this case $\gamma_{r_3}$ can just be a generic arc from $s_{r_3}$ to $g(s_{r_3})$ and it will have no intersections with $\overline\bdry_{c(i_2,i_3)}$.
\end{proof}

\s
We now explain how to modify $\mathfrak{S}$ to $\mathfrak{S}'$ near the codimension one and two ``boundaries" of $\V_i$ so that (M1) and (M2) hold.
In other words, we modify the sections
$$(\mathfrak{s}_{(i_1,i_2,i_3)},  \mathfrak{s}_{(i_1,c(i_2,i_3))}, \mathfrak{s}_{(c(i_1,i_2),i_3)}, \mathfrak{s}_{c(i_1,i_2,i_3)})$$ to
$$(\mathfrak{s}'_{(i_1,i_2,i_3)}, \mathfrak{s}'_{(i_1,c(i_2,i_3))}, \mathfrak{s}'_{(c(i_1,i_2),i_3)}, \mathfrak{s}_{c(i_1,i_2,i_3)}).$$
The modifications will take place on the set
$$X=\{\mathfrak{nl}_{i_1}\geq \mathcal{L'}\}\cup \{\mathfrak{nl}_{i_2}\geq \mathcal{L'}\},$$
where $\mathcal{L'}\gg \mathcal{L}$; in other words, $\mathfrak{s}_*=\mathfrak{s}'_*$ on the complement of $X$.
In the rest of this step we encourage the reader to refer to Figure~\ref{fig: corner structure} for the picture of a corner, where $i_1,i_2,i_3, c(i_1,i_2), c(i_2,i_3)$, $c(i_1,i_2,i_3)$ are labeled $1$--$6$.

First we define
\begin{align*}
\mathfrak{s}'_{(i_1,c(i_2,i_3))}=\left\{
\begin{array}{cl}
(\gamma^*_{r_1}(\mathfrak{nl}_{i_1}),\gamma^*_{r_3}(\mathfrak{nl}_{i_1})),  & \mbox{ for } \mathfrak{nl}_{i_1}\geq \mathcal{L}',\\
\mathfrak{s}_{(i_1,c(i_2,i_3))},  & \mbox{ for }  \mathfrak{nl}_{i_1}\leq \mathcal{L}'.
\end{array}\right.
\end{align*}
By Lemma~\ref{lemma: zero signed intersection 2}(1), the signed intersection number between $\overline\bdry_{(i_1,c(i_2,i_3))}$ and $\mathfrak{s}'_{(i_1,c(i_2,i_3))}$ on $\mathfrak{nl}_{i_1}\geq \mathcal{L}'$ is zero.

Next consider the pushforwards of $\mathfrak{s}_{(i_1,c(i_2,i_3))}$ and $\mathfrak{s}'_{(i_1,c(i_2,i_3))}$ under the morphism $\phi_{(i_1,c(i_2,i_3)),(i_1,i_2,i_3)}$. On the overlap
$$X_{3,0}:=\{\mathfrak{nl}_{i_1}\geq \mathcal{L}',\mathcal{L}-\varepsilon''\leq \mathfrak{nl}_{i_2}\leq \mathcal{L}\},$$
the section $\mathfrak{s}_{(i_1,c(i_2,i_3))}=(s_{r_1},s_{r_3})$ is sent to $\mathfrak{s}_{(i_1,i_2,i_3)}=(s_{r_1},0,s_{r_3})$ and the section $\mathfrak{s}'_{(i_1,c(i_2,i_3))}=(\gamma^*_{r_1}(\mathfrak{nl}_{i_1}),\gamma^*_{r_3}(\mathfrak{nl}_{i_1}))$ is sent to $\mathfrak{s}'_{(i_1,i_2,i_3)}=(\gamma^*_{r_1}(\mathfrak{nl}_{i_1}),0,\gamma^*_{r_3}(\mathfrak{nl}_{i_1}))$.  By applying Lemma~\ref{lemma: zero signed intersection 2}(2) to the term $\gamma^*_{r_3}(\mathfrak{nl}_{i_1})$, we see that $\overline\bdry_{(i_1,i_3,i_3)}$ has no intersections with $\mathfrak{s}'_{(i_1,i_2,i_3)}$ on $X_{3,0}$ if we take $\varepsilon''>0$ to be sufficiently small.

On $$X_{3,1}:=\{\mathfrak{nl}_{i_1}\geq \mathcal{L}',\mathcal{L}\leq \mathfrak{nl}_{i_2}\leq \mathcal{L}+\varepsilon''\},$$
\begin{equation} \label{eqn: 123}
\mathfrak{s}_{(i_1,i_2,i_3)}=(\zeta(\mathfrak{nl}_{i_1}) s_{r_1},\zeta(\mathfrak{nl}_{i_2}) s_{r_2},s_{r_3})= ( s_{r_1},\zeta(\mathfrak{nl}_{i_2}) s_{r_2},s_{r_3}).
\end{equation}
We then set
\begin{equation} \label{eqn: 123 prime}
\mathfrak{s}'_{(i_1,i_2,i_3)}=(\gamma^*_{r_1}(\mathfrak{nl}_{i_1}),\zeta(\mathfrak{nl}_{i_2}) \gamma^*_{r_2}(\mathfrak{nl}_{i_1}),\gamma^*_{r_3}(\mathfrak{nl}_{i_1})),
\end{equation}
where $\gamma_{r_2}:[0,1]\to B^{r_2,\ell}_{\rho_2}$ is some path from $s_{r_2}$ to $g(s_{r_2})$ with $\rho_2=2|s_{r_2}|$.

Now we come to the key point: $\mathfrak{s}_{(i_1,i_2,i_3)}$ and $\mathfrak{s}'_{(i_1,i_2,i_3)}$ do not intersect $\overline\bdry_{(i_1,i_2,i_3)}$ on $X_{3,1}$ for $\varepsilon''>0$ sufficiently small.  This is due to $|s_{r_2}|\ll |s_{r_3}|$ by Condition (\hyperlink{star}{*}). Since  $\gamma_{r_3}$ does not intersect $\overline\bdry_{c(i_2,i_3)}(\bdry\V_{c(i_2,i_3)})$ by Lemma~\ref{lemma: zero signed intersection 2}(2), $\overline\bdry_{(i_1,i_2,i_3)}|_{X_{3,1}}$ does not intersect a small neighborhood of $(\gamma^*_{r_1}(\mathfrak{nl}_{i_1}),0, \gamma^*_{r_3}(\mathfrak{nl}_{i_1}))$.  In particular, if $|s_{r_2}|$ is sufficiently small, then $\overline\bdry_{(i_1,i_2,i_3)}|_{X_{3,1}}$ does not intersect $\mathfrak{s}'_{(i_1,i_2,i_3)}$; $\mathfrak{s}_{(i_1,i_2,i_3)}$ is similar.

The situation for $\mathfrak{s}'_{(c(i_1,i_2),i_3)}$ and $\mathfrak{s}'_{(i_1,i_2,i_3)}$ on
$$X_{0,3}\cup X_{1,3}:=\{\mathfrak{nl}_{i_2}\geq \mathcal{L}',\mathcal{L}-\varepsilon'' \leq \mathfrak{nl}_{i_1}\leq \mathcal{L}+\varepsilon''  \}$$
is analogous (for one of Steps 3A, 3B, or 3C).

It remains to modify $\mathfrak{s}_{(i_1,i_2,i_3)}$ to $\mathfrak{s}'_{(i_1,i_2,i_3)}$ on $X_{3,2}\cup X_{2,3}\cup X_{3,3}$, where:
\begin{gather*}
X_{3,2}:=\{\mathfrak{nl}_{i_1}\geq \mathcal{L}', \mathcal{L}+\varepsilon''\leq \mathfrak{nl}_{i_2}\leq \mathcal{L}'\},\\ X_{2,3}:=\{\mathcal{L}+\varepsilon''\leq \mathfrak{nl}_{i_1}\leq  \mathcal{L}', \mathfrak{nl}_{i_2}\geq \mathcal{L}'\},\\
X_{3,3}:=\{\mathfrak{nl}_{i_1}, \mathfrak{nl}_{i_2}\geq \mathcal{L}'\}.
\end{gather*}
On $\{\mathfrak{nl}_{i_1},\mathfrak{nl}_{i_2}\geq \mathcal{L}+\varepsilon''\}$ we have $\mathfrak{s}_{(i_1,i_2,i_3)}=(s_{r_1},s_{r_2},s_{r_3})$.  We then define $\mathfrak{s}'_{(i_1,i_2,i_3)}$ as:
\be
\item $(\gamma^*_{r_1}(\mathfrak{nl}_{i_1}),\gamma^*_{r_2}(\mathfrak{nl}_{i_1}),\gamma^*_{r_3}(\mathfrak{nl}_{i_1}))$ on $X_{3,2}$;
\item $(\gamma^*_{r_1}(\mathfrak{nl}_{i_2}),\gamma^*_{r_2}(\mathfrak{nl}_{i_2}),\gamma^*_{r_3}(\mathfrak{nl}_{i_2}))$ on $X_{2,3}$;
\item $(\gamma^*_{r_1}(\beta(\mathfrak{nl}_{i_1},\mathfrak{nl}_{i_2})),\gamma^*_{r_2} (\beta(\mathfrak{nl}_{i_1},\mathfrak{nl}_{i_2})),\gamma^*_{r_3}(\beta(\mathfrak{nl}_{i_1},\mathfrak{nl}_{i_2})))$ on $X_{3,3}$, where $$\beta(a,b)=\sqrt{(a-\mathcal{L}')^2+(b-\mathcal{L}')^2}-\mathcal{L}'.$$
\ee
The images of the maps (1)--(3) are $1$-dimensional, since each is a postcomposition by $(\gamma^*_{r_1},\gamma^*_{r_2},\gamma^*_{r_2})$, which has $1$-dimensional image. On the other hand, by Equation~\eqref{eqn: sum of dimensions}, two of the three maps $\overline\bdry_{i_j}: \V_{i_j}\to e^{r_j,\ell}$, $j=1,2,3$, have codimension at least one or one of the maps has codimension at least two.   Hence if $\gamma_{r_j}$, $j=1,2,3$, are sufficiently generic, then $\mathfrak{s}_{(i_1,i_2,i_3)}$ and $\mathfrak{s}'_{(i_1,i_2,i_3)}$ have no intersections with $\overline\bdry_{(i_1,i_2,i_3)}$ on $X_{3,2}\cup X_{2,3}\cup X_{3,3}$.

\s\n
{\em Step 3B.}
Suppose that $\op{vdim} \mathcal{M}_{c(i_2,i_3)}=0$.  Then $\op{vdim}\mathcal{M}_{i_1}=-1$.  The only differences with Step 3A are that, assuming genericity of $\gamma_{r_1}$ and $\gamma_{r_3}$:
\begin{itemize}
\item $\gamma_{r_1}: [0,1]\to e^{r_1,\ell}$ satisfies the conditions of Lemma~\ref{lemma: zero signed intersection} (where we replace $i_2$ by $i_1$) and intersects $\overline\bdry_{i_1}$ at isolated points;
\item $\gamma_{r_3}:[0,1]\to e^{r_3,\ell}$ intersects $\overline\bdry_{c(i_2,i_3)}|_{\bdry \V_{c(i_2,i_3)}}$ at isolated points since it is a codimension one map; and
\item the intersection points do not occur at the same time in $[0,1]$.
\end{itemize}
It implies that if $|s_{r_2}|\ll |s_{r_3}|$, then $\mathfrak{s}_{(i_1,i_2,i_3)}$ and $\mathfrak{s}'_{(i_1,i_2,i_3)}$, given by Equations~\eqref{eqn: 123} and \eqref{eqn: 123 prime}, have no intersections with $\overline \partial_{(i_1,i_2,i_3)}$ on $X_{3,0}\cup X_{3,1}$.

\s\n
{\em Step 3C.}
Suppose that $\op{vdim} \mathcal{M}_{c(i_2,i_3)}\geq 1$. Then $\op{vdim}\mathcal{M}_{i_1}\leq -2$ and
\begin{itemize}
\item $\gamma_{r_1}:[0,1]\to e^{r_1,\ell}$ does not intersect $\op{Im}\overline\bdry_{i_1}$.
\end{itemize}
If $|s_{r_2}|\ll |s_{r_3}|$, then $\mathfrak{s}_{(i_1,i_2,i_3)}$ and $\mathfrak{s}'_{(i_1,i_2,i_3)}$, given by Equations~\eqref{eqn: 123} and \eqref{eqn: 123 prime}, have no intersections with $\overline \partial_{(i_1,i_2,i_3)}$ on $X_{3,0}\cup X_{3,1}$.

\s
This implies the theorem for $m=3$.  The general case is completely analogous and is only more complicated in notation.
\end{proof}

\section{Equivariant Lagrangian Floer cohomology}\label{Equivariant Lagrangian Floer cohomology}

\subsection{Grading} \label{subsection: grading}

In order to $\Z$-grade our equivariant Lagrangian Floer cohomology groups, we require $L_0$ and $L_1$ to be {\em $G$-equivariantly graded}, i.e., (G1)--(G3) to hold.
\be
\item[(G1)] $c_1(M,J) = 0$.
\ee
Then there exists a nowhere-vanishing section $\Omega$ of $\wedge_{\C} ^n (T^*M, J)$.
Let $\op{det}_{\Omega, i}^2: L_i \to S^1$ be the map given by
$$\op{det}_{\Omega, i}^2 (p_i) = \frac{\Omega (X_{i,1}\wedge \dots \wedge X_{i,n})^2}{|\Omega (X_{i,1}\wedge\dots \wedge X_{i,n})|^2},$$
where $p_i \in L_i$ and $X_{i,1},\dots , X_{i,n}$ is a basis of $T_{p_i} L_i$.
\be
\item[(G2)] There exists a function $\theta_i: L_i \to \R $ that lifts $\op{det}_{\Omega, i}^2$, i.e.,
$$e^{2\pi \sqrt{-1} \theta_i(p_i)} = \op{det}_{\Omega, i}^2 (p_i).$$
\ee
Then for each $p\in L_0 \cap L_1$, we define an integer index $\mu(p)$ by the formula
$$\mu (p) = n + \theta_1 - \theta_0 -2 \angle (T_p L_0, T_p L_1).$$
Here $\angle(T_pL_0, T_pL_1) = \alpha_1 + \cdots + \alpha_n$, where the $\alpha_i \in (0,\tfrac{1}{2})$ are defined by choosing a unitary basis $\{ e_1,\dots, e_n \}$ of $T_pL_0$ with respect to $\omega$ and $J$ and writing
$$T_p L_1 = \R\langle e^{2\pi \sqrt{-1} \alpha_1} e_1, \dots , e^{2\pi \sqrt{-1} \alpha_n} e_n\rangle.$$
\be
\item[(G3)] $\mu(gp) = \mu(p)$ for all $p\in L_0\cap L_1$ and $g\in G$.
\ee

For more details on grading, we refer the reader to \cite{Se2} or \cite[Section 2.3]{AB}.

\subsection{Chain complex} \label{subsection: chain complex}

\cbu
Recall the Novikov ring
$$R = \Big\{\sum_{i=0}^\infty a_i T^{\lambda_i} ~|~  a_i \in \Z, \lambda_i \in \R_{\geq 0}, \lambda_0 = 0 \text{ and }  \lim_{i\to \infty} \lambda_i = \infty \Big\},$$
where $T$ is the formal parameter. We define the $\Z$-graded $R$-module
\begin{gather*}
CF^{\bullet}(L_0,L_1) = \textstyle\bigoplus_{j} CF^j(L_0, L_1),\\
CF^{j}(L_0,L_1)=R\langle p \in L_0 \cap L_1 ~|~ \mu(p) = j\rangle.
\end{gather*}
\cb
The differential
$$d: CF^{j}(L_0,L_1)\to CF^{j+1}(L_0,L_1)$$
is defined on the generators by \cbu 
$$d[p]=\sum\# \mathcal Z(\Kur(p,q;A),\mathfrak{S})\cdot T^{\int_A\omega} [q],$$ \cb
where the sum is over all $q\in L_0\cap L_1$ and $A\in \pi_2(p,q)$ subject to the condition $\op{vdim}(\mathcal{M}(p,q;A))=0$.

\begin{lemma} \label{lemma: d squared}
$d^2=0$.
\end{lemma}

\begin{proof}[Sketch of proof]
We consider the ends of the $1$-manifold $\mathcal{Z}(\Kur (p,q;A),\mathfrak{S})$ where $\op{vdim}(\mathcal{M}(p,q;A))=1$.  First observe that, by codimension reasons, for any $u\in \mathcal{Z}(\Kur (p,q;A),\mathfrak{S})$, we have $\mathfrak{nl}_r(u)\leq \mathcal{L}+\varepsilon''$ for all but possibly one $r\in L_0\cap L_1$ (cf.\ Equation~\eqref{eqn: defn of s I} for the definition of the section $\mathfrak{s}_I$ and Definition~\ref{defn: zeta} for the definition of $\zeta$ and note that $\zeta([\mathcal{L}+\varepsilon'',\infty))=1$).  Hence the ends of $\mathcal{Z}(\Kur (p,q;A),\mathfrak{S})$ are in bijection with
$$\coprod_{r} \mathcal{Z}(\Kur (p,r;A_1),\mathfrak{S}_1)\times \mathcal{Z}(\Kur(r,q;A_2),\mathfrak{S}_2),$$
where $\op{vdim}(\mathcal{M}(p,r;A_1))=\op{vdim}(\mathcal{M}(r,q;A_2)=0$ and $\mathfrak{S}_1$ and $\mathfrak{S}_2$ are compatible with $\mathfrak{S}$.
\end{proof}

We define the usual Lagrangian Floer cohomology from $L_0$ to $L_1$ by $$HF^{\bullet} (L_0, L_1) = \ker d/  \op{Im} d.$$

By Theorem~\ref{thm: equivariant count}, if $\op{vdim}(\mathcal{M}(p,q;A))=0$ and \cbu $g\in G$ takes $\Kur(p,q;A)$ to itself,\cb then
\begin{equation*} \label{independent of section}
\# \mathcal{Z}(\Kur(p,q;A),\mathfrak{S})=\# \mathcal{Z}(\Kur(p,q;A), g(\mathfrak{S})).
\end{equation*}
Hence,
\begin{align*}
dg[p] = & \sigma(g,p)d[gp] \\
= & \sigma(g,p) \sum_{r,B} \# \mathcal{Z}(\Kur(gp,r;B),\mathfrak{S}) [r] \\
= & \sigma(g,p) \sum_{q,A} \# \mathcal{Z}(\Kur(gp,gq;gA),\mathfrak{S}) [gq] \\
= & \sigma(g,p) \sum_{q,A} \# \mathcal{Z}(\Kur(gp,gq;gA),g\mathfrak{S}) [gq] \\
= & \sigma(g,p) \sum_{q,A} \#( g(\mathcal{Z}(\Kur(p,q;A),\mathfrak{S})) [gq] \\
= &  \sigma(g,p) \sum_{q,A} \# \mathcal{Z}(\Kur(p,q;A),\mathfrak{S})\sigma(g,p)\sigma(g,q) [gq]\\
= & gd[p],
\end{align*}
i.e., $d$ is $R[G]$-linear.

Let $(P_\bullet, d_P)$ be a projective resolution of $R$ over $R[G]$.
We denote
$$E^{i,j} := \op{Hom}_{R[G]}(P_{i}, CF^j(L_0,L_1)),$$
where $P_{i}$ and $CF^j(L_0,L_1)$ are regarded as $R[G]$-modules.
Let $d_>^{i,j}: E^{i,j} \to E^{i+1,j}$ be the map induced by $d_{P}: P_{i+1} \to P_{i}$,
and $d_\wedge ^{i,j}: E^{i,j} \to E^{i,j+1}$ be the map induced by $d: CF^j(L_0,L_1) \to CF^{j+1}(L_0,L_1)$ multiplied by the factor $(-1)^i$.
Then $d_>^{i,j}$ and $d_\wedge^{i,j}$ commute with the multiplication by elements in $R[G]$ and form a double complex.

The {\em $G$-equivariant Lagrangian Floer cochain complex} is the total complex
$$(CF^{\bullet}_G (L_0,L_1), d_G),$$
where
$$CF^k_G(L_0,L_1) =\textstyle \bigoplus_{i+j = k} E^{i,j}, \quad d_G|_{E^{i,j}}  = d_>^{i,j} + d_\wedge^{i,j}.$$
The corresponding $G$-equivariant Lagrangian Floer cohomology group is:
$$HF_G^{\bullet}(L_0,L_1)= \ker d_G / \op{Im} d_G.$$

The cohomology $H^\bullet (\op{Hom}_{R[G]}(P_\bullet, R))\cong H^\bullet (BG)$  (taking $Y = \text{\{pt\}}$ as in Section~\ref{section: introduction}) is a ring whose product is the standard cup product.
Similarly we can define the following $R[G]$-bilinear map: 
$$H^\bullet (BG) \times HF_G^{\bullet}(L_0,L_1) \to HF_G^{\bullet}(L_0,L_1),$$
which makes $ HF_G^{\bullet}(L_0,L_1)$ an $H^\bullet (BG)$-module. \cbu
Indeed, it is easier to see the module structure via the definition of $HF_G^\bullet(L_0, L_1)$ using the singular chain complex $C_\bullet (EG)$ in place of $P_\bullet$. \cb
More precisely, the product on the chain level is induced by the composition of the K\"{u}nneth map followed by the diagnal map:
\begin{align*}
\op{Hom}_{R[G]}(C_{i}(EG), CF^j(L_0,L_1)) \times \op{Hom}_{R[G]}(C_{k}(EG), R) \\
 \overset{\text{K\"{u}nneth map}}{\longrightarrow} \op{Hom}_{R[G]}(C_{i+k}(EG \times EG), CF^j(L_0,L_1))  \\
 \overset{\Delta^*}{\longrightarrow} \op{Hom}_{R[G]}(C_{i+k}(EG), CF^j(L_0,L_1)).
\end{align*}

\cbu 
From standard results on spectral sequences of double complexes, we obtain:
\cb
\begin{lemma}
There exists a spectral sequence $\{E^{i,j}_r\}_r$ with second page
$$E^{i,j}_2 = \op{Ext}^i_{R[G]}(R, HF^{j}(L_0,L_1))$$
converging to $HF_G^{\bullet}(L_0, L_1)$.
\end{lemma}

\subsection{Chain map}

Using the notation from Section~\ref{subsection: chain map and chain homotopy}, for $p\in L_0\cap L_1$, $q\in L_0'\cap L_1'$, and $A\in \pi_2(p,q)$, there exists a Kuranishi structure $\Kur(p,q;A)$ and a collection of sections $\mathfrak{S}$ such that chain map
$$\Phi: CF^{\bullet}(L_0,L_1)\to CF^{\bullet}(L_0',L_1')$$
is defined on the generators  by \cbu
$$\Phi (p)= \sum \#\mathcal Z(\Kur(p,q;A),\mathfrak{S})\cdot T^{\int_A\omega} q,$$ \cb
where the sum is over all $q\in L_0'\cap L_1'$ and $A\in \pi_2(p,q)$ subject to the conditions $\op{vdim}(\mathcal{M}^\circ(p,q;A))=0$ and \cbu $(p,q,A)$ belongs to Sequence (\ref{ordering of triples})\cb. We also have the following, whose proof is similar to that of Lemma~\ref{lemma: d squared}:

\begin{lemma}
$d\circ\Phi=\Phi\circ d$.
\end{lemma}

The proof of Theorem~\ref{thm: equivariant count} carries over for chain maps.  In other words, if $\op{vdim}(\mathcal{M}^\circ(p,q;A))=0$ and $g\in G$ preserves $\Kur(p,q;A)$, then
$$\# \mathcal{Z}(\Kur(p,q;A),\mathfrak{S})=\# \mathcal{Z}(\Kur(p,q;A), g(\mathfrak{S})).$$
This implies that:

\begin{lemma} \label{lemma: chain map}
$\Phi:CF^{\bullet}(L_0,L_1)\to CF^{\bullet}(L_0',L_1')$ is a chain map of $R[G]$-modules.
\end{lemma}

It is clear from the definition that the chain map $\Phi$ induces the chain map
$$\Phi_G: CF_G^{\bullet}(L_0, L_1) \to CF_G^{\bullet} (L_0', L_1').$$

\subsection{Chain homotopy} \label{subsection: chain homotopy}

Let
$$\Phi:CF^{\bullet}(L_0,L_1)\to CF^{\bullet}(L_0',L_1') \quad \mbox{ and } \quad \Psi:CF^{\bullet}(L_0',L_1')\to CF^{\bullet}(L_0,L_1)$$
be chain maps of $R[G]$-modules, defined using $\phi_s$ and $\phi_s^{-1}$.

Fix $p\in L_0\cap L_1$, $q\in L_0'\cap L_1'$, and $A\in \pi_2(p,q)$.  Using the function $\Theta$ from Section~\ref{subsubsection: chain homotopies} we construct the bundles $\pi_{I,[0,1]}: \E_{I,[0,1]}\to \V_{I,[0,1]}$ and the $1$-parameter family
$$\Kur_{[0,1]}(p,q;A):= \coprod_{\tau\in[0,1]} \Kur_\tau(p,q;A),$$
of Kuranishi structures. Here each term $ \Kur_\tau(p,q;A)$ corresponds to $\Kur(p,q;A)$ for $\tau\in[0,1]$. We can take the sections $\mathfrak{s}_{I,[0,1]}$ of $\mathfrak{S}_{[0,1]}$, viewed as a map to a fixed vector space, to be ``independent of $\tau$'' or, more precisely, only dependent on neck lengths.

Now consider the ends of $\mathcal{Z}(\Kur_{[0,1]} (p,q;A),\mathfrak{S}_{[0,1]})$ for $\op{vdim}(\mathcal{M}(p,q;A))=0$. A similar argument as Lemma~\ref{lemma: chain map} gives

\begin{lemma}\label{lemma: chain homotopy}
$\Phi\circ \Psi - \op{id}= K\circ d +d\circ K,$ where
$$K: CF^{\bullet} (L_0, L_1) \to CF^{\bullet} (L_0, L_1)$$
is a map of $R[G]$-modules.
\end{lemma}

Let $K_G: CF_G^{\bullet} (L_0, L_1) \to CF_G^{\bullet} (L_0, L_1)$ be the map induced by $K$. We obtain
$$\Phi_G\circ \Psi_G - \op{id} = K_G\circ d_G + d_G\circ K_G.$$

Summarizing, we have:

\begin{cor}
The $G$-equivariant Lagrangian Floer cohomology $HF_G^{\bullet} (L_0, L_1)$ is independent of the choice of equivariant almost complex structure $J$ and is an invariant of the pair $(L_0, L_1)$ under $G$-equivariant Hamiltonian isotopy.
\end{cor}

\end{document}